\documentclass[11pt]{amsart}
\usepackage[margin=1in]{geometry}  % package to make nice margins
\usepackage{longtable}
\usepackage[table]{xcolor}
\usepackage{amsmath, amsthm, latexsym, amssymb, amsfonts}
\usepackage{alltt}
\usepackage[pdftex]{graphicx} 
\usepackage[all]{xy}

 \theoremstyle{plain}
 \newtheorem{tm}{Theorem}[section]
 \newtheorem{lm}[tm]{Lemma}

 \newtheorem{pro}[tm]{Proposition}

 \theoremstyle{definition}
 \newtheorem{defi}[tm]{Definition}
 \newtheorem{rema}[tm]{Remark}
 
  \newtheorem{ex}[tm]{Example}
  \newtheorem{pdr}[tm]{Procedure}
  
 \newtheorem*{th*}{Theorem}

\newcommand{\cl}[1]{\mathcal{#1}}

\newcommand{\Z}{\mathbb Z}
\newcommand{\N}{\mathbb N}
\newcommand{\R}{\mathbb R}
\newcommand{\F}{\mathbb F}
\newcommand{\Pp}{\mathbb P}

\newcommand{\cA}{\cl A}

\newcommand{\cC}{\cl C}

\newcommand{\K}{{\mathbb K}}

\newcommand{\la}{\langle}
\newcommand{\ra}{\rangle}

\newcommand{\m}{\mathfra\K{m}}

\newcommand{\reg}{\operatorname{reg}}
\newcommand{\Cl}{\operatorname{Cl}}

    \newcommand{\rank}{\operatorname{rank}}

\newcommand{\dis}{\displaystyle}
\def\aa{{\bf \alpha}}
\def\bb{\beta}

\def\t{{\bf t}}
\def\q{{\bf q}}
\def\y{{\bf y}}

\def\uu{{\mathbf{u}}}
\def\vv{{\bf v}}
\def\x{{\bf x}}
\def\q{{\bf q}}
\def\m{{\bf m}}
\def\kk{{\mathbf k}}
\def\hh{{\mathbf h}}

\def\ev{{\text{ev}}}

 \usepackage{graphics}
 \usepackage[english]{babel}
\usepackage[latin1]{inputenc}
\usepackage{times}
\usepackage[T1]{fontenc}
 \usepackage{graphics}
\usepackage{algorithm}
\usepackage{algpseudocode}

\begin{document}

%%%%%%%%%%%%%%%%%%%%%%%Topmatter

\title{On parameterized toric codes}
 \thanks{The first author is supported by T\"{U}B\.{I}TAK-2211, the second author is supported by T\"{U}B\.{I}TAK Project No:114F094}

\author[Esma Baran]{Esma Baran}
\address[Esma Baran]{Department of Mathematics, \c{C}ank{\i}r{\i} Karatekin University, \c{C}ank{\i}r{\i}, TURKEY}
\email{esmabaran@karatekin.edu.tr}
\author[Mesut \c{S}ah\.{i}n]{Mesut \c{S}ah\.{i}n}
\address[Mesut \c{S}ah\.{i}n]{Department of Mathematics, Hacettepe  University, Ankara, TURKEY}
\email{mesut.sahin@hacettepe.edu.tr}
%\author[Ivan Soprunov]{Ivan Soprunov}
%\address[Ivan Soprunov]{Department of Mathematics, Cleveland State University, Cleveland, OH USA}
%\ead{i.soprunov@csuohio.edu}

%\date{}

\begin{abstract} Let $X$ be a complete simplicial toric variety over a finite field with a split torus $T_X$. For any matrix $Q$, we are interested in the subgroup $Y_Q$ of $T_X$ parameterized by the columns of $Q$. We give an algorithm for obtaining a basis for the unique lattice $L$ whose lattice ideal $I_L$ is $I(Y_Q)$. We also give two direct algorithmic methods to compute the order of $Y_Q$, which is the length of the corresponding code ${\cC}_{\aa,Y_Q}$. We share procedures implementing them in \verb|Macaulay2|. Finally, we give a lower bound for the minimum distance of ${\cC}_{\aa,Y_Q}$, taking advantage of the parametric description of the subgroup $Y_Q$. As an application, we compute the main parameters of the toric codes on Hirzebruch surfaces $\cl H_{\ell}$ generalizing the corresponding result given by Hansen. 
\end{abstract}

\maketitle

%%%%%%%%%%%%%%%%%%%%%%%%%%%%%%%%%%%%%%%%%%%%%%%%%%%%%%%%%%
\section{Introduction} Let $X$ be a complete simplicial toric variety over a finite field $\K=\F_q$ with a split torus $T_X\cong (\K^*)^{n}$. Our main goal in the present paper is to uncover some algebraic and geometric properties of subgroups $Y_Q=\{[{\t}^{\q_{1}}:\cdots:{\t}^{\q_{r}}]|\t\in (\K^{*})^s\}$
of the algebraic group $T_X$, and develop techniques applying to certain algebraic-geometric codes, for any matrix $Q=[\q_1 \q_2\cdots \q_r]\in M_{s\times r}(\Z)$. It is known that all subgroups of $T_X$ are of this form by \cite[Theorem 3.2 and Corollary 3.7]{Sahin}. 

Denote by $S=\K[x_1,\dots,x_r]$ the homogeneous coordinate ring of $X$, which is $\Z^d$-graded. If $S_{\aa}$ is the finite dimensional vector space spanned by the monomials in $S$ having degree $\aa$, then evaluating polynomial functions from $S_{\aa}$ at the points $[P_1],\dots,[P_N]$ of $Y_Q$ defines the following $\K$-linear map $$\ev_{Y_Q}:S_\aa\to \K^N,\quad F\mapsto (F(P_1),\dots,F(P_N)).$$  
The image $\text{ev}_{Y_Q}(S_\aa) \subseteq \F_q^N$ is a linear code which is denoted by ${\cC}_{\aa,Y_Q}$ and is called the \textbf{parameterized toric code} associated to $Q$. There are $3$ main parameters $[N,K,\delta]$ of a linear code. The \textit{length} $N$ of ${\cC}_{\aa,Y_Q}$ is the order $|Y_Q|$ of the subgroup in our case. The \textit{dimension} of ${\cC}_{\aa,Y_Q}$, denoted $K=\dim_{\K}({\cC}_{\aa,Y_Q})$, is the dimension as a subspace of $\F_q^N$. The number of non-zero entries in any $c\in{{\cC}_{\aa,Y_Q}}$ is called its \textit{weight} and  \textit{minimum distance} $\delta$ of ${\cC}_{\aa,Y_Q}$ is the smallest weight among all code words $c\in{{\cC}_{\aa,Y_Q}}\setminus\{0\}$. 

Parameterized toric codes includes toric codes, constructed by Hansen in \cite{Ha0}, as a special case where $Q$ is the identity matrix $I_r$ and $Y_Q$ is the full torus $T_X$. Toric codes are among evaluation codes on a toric variety showcasing champion examples, see \cite{BRorderbound,BK7codes,Little13}. The length in this special case is $|T_X|=(q-1)^n$. When the evaluation map is injective, the dimension is the number of monomials of degree $\alpha$. Computing the minimum distance is a very challenging task which have been completed in some special cases, in contrast to more general situations where some lower bounds and/or upper bounds on the minimum distance have been given via different methods, see \cite{Ha1,Jo,Little17,LiSche,LiSchw,Ru,SoSo1}. 

There is an algebraic approach for studying these codes relying on the vanishing ideal $I(Y_Q)$ of $Y_Q$ which is the graded ideal generated by homogeneous polynomials in $S$ vanishing at every point of $Y_Q$. Since the kernel of the linear map $\ev_{Y_Q}$ equals the homogeneous piece $I(Y_Q)_\aa$ of degree $\aa$, we have an isomorphism of $\K$-vector spaces $S_\aa/I(Y_Q)_\aa \cong{\cC}_{\aa,Y_Q}$. Thus, the dimension of ${\cC}_{\aa,Y_Q}$ is the multigraded Hilbert function $H_{Y_Q}(\aa):=\dim_{\F_q} S_{\aa}-\dim_{\F_q} I(Y_Q)_\aa$ of $I(Y_Q)$. Initially, there are infinitely many codes corresponding to elements in the semigroup $\N\beta:=\N\bb_1+\cdots+\N\bb_r$, where $\bb_i=\deg (x_i)$ for $i=1,\dots,r$. Since these codes are subspaces of the space $\F_q^N$, the upper bound for the dimension $H_{Y_Q}(\aa)$ of ${\cC}_{\aa,Y_Q}$ is exactly $N=|Y_Q|$. By Singleton's bound $\delta\leq N+1-K$, the minimum distance attains its minimum value $1$ when the dimension $K$ reaches its upper bound $N$.  An important algebraic invariant of $Y_Q$ in detecting these trivial codes is the so-called \textit{multigraded regularity} defined by $$\reg(Y_Q):=\{\aa\in \N\bb \quad:\quad H_{Y_Q}(\aa)= |Y_Q|\}\subseteq \N^d.$$ So, non-trivial codes come from the set $\N\bb \setminus \reg(Y_Q)$. There are also equivalent codes having the same parameters which can be detected using the values of the Hilbert function. More precisely, the codes ${\cC}_{\aa,Y_Q}$ and ${\cC}_{\aa',Y_Q}$ are equivalent if $H_{Y_Q}(\aa)=H_{Y_Q}(\aa')$ whenever $\aa-\aa'\in \N\bb$, and hence, there are only finitely many interesting codes on each variety $X$, for a fixed matrix $Q$ and prime power $q$ by \cite[Proposition 4.3]{sasop}. The core of this approach is to use the ideal $I(Y_Q)$ for determining these finitely many codes before constructing any code. In order to determine elements $\aa$ corresponding to them, we need to determine $|Y_Q|$ first, obtain a generating system of $I(Y_Q)$ and then analyze the values of the Hilbert function of $I(Y_Q)$, see Example \ref{P(2,2,3,5)}. This yields a finite list of interesting codes together with their lengths and dimensions. The minimum distance can also be computed using the ideal $I(Y_Q)$, see \cite{MinDisVillarrealJPAA} if $X$ is a projective space. When $I(Y_Q)$ is a complete intersection, lower bounds for the minimum distance of ${{\cC}_{\aa,Y_Q}}$, can be computed using \cite[Theorem 3.2 and Theorem 3.9]{sop}. These motivate developing methods and algorithms for computing a generating set of the vanishing ideal $I(Y_Q)$ and checking if it is a complete intersection. 
 
Parameterized codes were defined and studied for the first time by Villarreal, Simis and Renteria in \cite{Algebraicmethodsforparameterizedcodesandinvariantsofvanishingoverfinitefields} when $X$ is a projective space. Among other interesting results, they gave a method for computing a generating set of $I(Y_Q)$ and showed that $I(Y_Q)$ is a lattice ideal of dimension $1$. Later, the lattice of the vanishing ideal is determined more explicitly, when $Q$ is a diagonal matrix in \cite{LVZ}. When $Y_Q$ is the torus $T_X$ lying in the projective space $X=\Pp^n$, that is $Q=I_r$, the main parameters are determined in \cite{SPV}. Dias and Neves generalized parametrized codes from standard projective space to weighted projective spaces $\Pp(w_1,\dots,w_r)$, and showed that the vanishing ideal of the torus $T_X$ is a lattice ideal of dimension $1$ in \cite{Codesoveraweightedtorus}. 

In the first part of the present paper, we use some of the ideas in these papers to extend them into the more general setting of a toric variety. Namely, Section \ref{S:LatticeBasis} gives a very useful description of the unique lattice $L$ whose ideal $I_L$ is nothing but $I(Y_Q)$, see Lemma \ref{lm:equality}. So, a generating set of $I(Y_Q)$ can be obtained from a basis of $L$. Theorem \ref{t:mod} gives a practical description of the lattice $L$ for which $I(Y_Q)=I_L$, leading to Algorithm \ref{a:lattice1}. We include Procedure \ref{pdr:lattice1} implementing this algorithm in \verb|Macaulay2| \cite{Mac2} for computing a basis for $L$. Thus, we can check if $I(Y_Q)$ is a complete intersection easily from this basis, see Remark \ref{rem:CI} and Example \ref{Ex:CI}.

 Section \ref{S:length} gives a direct method for computing the size of $Y_Q$ taking advantage of its parametric representation and giving the length of ${\cC}_{\aa,Y_Q}$. Our second method is inspired from \cite[Proposition 3.3]{Algebraicmethodsforparameterizedcodesandinvariantsofvanishingoverfinitefields} and gives a polytope whose lattice points determine the size of $Y_Q$, extending the corresponding result from the projective space to a general toric variety. However, our polytope is simpler than the polytope given in the special case where $X=\Pp^n$, see Remark \ref{R:polytope} and Example \ref{E:polytope}. 
 
 The main contribution of the paper is Section \ref{S:minimumdistance} in which we give a lower bound for the minimum distance of ${\cC}_{\aa,Y_Q}$, taking advantage of the parametric description of the subgroup $Y_Q$. As an application, we compute the main parameters of the toric codes on Hirzebruch surfaces in Theorem \ref{T:codesOnHirzebruch} generalizing the corresponding result in \cite{Ha1}. We also share a few examples in Section \ref{S:Ex} to reveal the potential of the family of parameterised codes.

\section{Preliminaries}
Let $\K=\F_q$ be a fixed finite field and $\Sigma\subset \R^n$ be a complete simplicial fan with rays $\rho_1,\dots,\rho_r$ generated by the  primitive lattice vectors $\vv_1,\dots,\vv_r\in \Z^n$, respectively. We consider the corresponding toric variety $X$ with a split torus $T_X\cong {(\K^*)^n}$. We assume that the class group $\Cl(X)$ have no torsion. Smooth $X$ with an $n$-dimensional
cone in its fan will satisfy this condition by  \cite[Proposition 4.2.5]{CLSch}. For applications to coding theory smooth toric varieties are sufficient, although we may prefer to study singular varieties such as weighted projective spaces for their simplicity. Given an element $\textbf{a}=(a_{1},\dots,a_{s})\in\Z^s$ we use $\textbf{t}^\textbf{a}$ to denote $t_1^{a_1}\cdots t_s^{a_s}$. Recall the construction of $T_X$ as a geometric quotient (see \cite{Coxhom} and \cite{CLSch}) via the following two key dual exact sequences:
$$\dis \xymatrix{ \mathfrak{P}: 0  \ar[r] & \Z^n \ar[r]^{\phi} & \Z^r \ar[r]^{{\bb}} & \cA \ar[r]& 0},$$  
where $\phi$  denotes  the matrix  $[\vv_1\cdots \vv_r]^T$ and $\cA=\Z^d\cong \Cl X$ for $d=r-n$,
$$\dis \xymatrix{ \mathfrak{P}^*: 1  \ar[r] & \mathcal{G} \ar[r]^{i} & (\K^*)^r \ar[r]^{\pi} & (\K^*)^n \ar[r]& 1},$$
where $\pi:(t_1,\dots, t_r)\mapsto (\mathbf{t}^{\uu_1}, \dots , \mathbf{t}^{\uu_n}),$ with $\uu_1,\dots, \uu_n$ being the columns of $\phi$ and $\mathcal{G}=\ker(\pi)$. Thus, $\uu_1,\dots, \uu_n$ constitute a natural $\Z$-basis for the key lattice $L_{\beta}=\ker \beta= \phi(\Z^n) \subset \Z^r$. The exact sequence $ \mathfrak{P}^*$ gives $T_X$  a quotient representation  $T_X \cong {(\K^*)^n}\cong(\K^*)^r /\mathcal{G}$, meaning that every element in the torus $T_X$ can be represented as $[p_1:\cdots:p_r]:=\mathcal{G}\cdot (p_1,\dots,p_r)$  for some $(p_1,\dots,p_r)\in (\K^*)^r$.

Denote by $S=\K[x_1,\dots,x_r]$ the homogeneous coordinate ring of $X$, which is $\Z^d$-graded by letting $\deg_{\cA}(x_j):=\beta_j:=\beta(e_j)$ using the exact sequence $ \mathfrak{P}$. Thus, $S=\bigoplus_{\aa \in \cA} S_{\aa}$, where $S_{\aa}$ is the finite dimensional vector space spanned by the monomials having degree $\aa$. Moreover, by \cite[Theorem 8.6 and Corollary 8.8]{CombinatorialCommutativeAlgebra}, one can choose $\beta_j\in \N^d$, where $\N$ is the set of non-negative integers. 

 \begin{ex}
\label{ex:Hirzebruch} 	
Let $X=\cl H_{\ell}$ be the Hirzebruch surface corresponding to a fan in $\R^2$ with primitive ray generators $\vv_1=(1,0)$, $\vv_2=(0,1)$, $\vv_3=(-1,\ell)$,
and $\vv_4=(0,-1)$, for any positive integer $\ell$. If $\uu_1=(1,0,-1,0)$, $\uu_2=(0,1,\ell,-1)$ and $\beta=\begin{bmatrix}
1 & 0 & 1& \ell\\
0 & 1 & 0& 1  
\end{bmatrix}$, then we have the following exact sequences 
$$\dis \xymatrix{ \mathfrak{P}: 0  \ar[r] & \Z^2 \ar[r]^{\phi} & \Z^4 \ar[r]^{\beta}& \cA \ar[r]& 0},$$
  
where $\phi=[\uu_1 \: \: \uu_2]$ and $L_{\beta}=\la \uu_1, \uu_2\ra$,
$$\dis \xymatrix{ \mathfrak{P}^*: 1  \ar[r] & \mathcal{G} \ar[r]^{i} & (\K^*)^{4} \ar[r]^{\pi} & (\K^*)^2 \ar[r]& 1}$$
where $\pi:\t\mapsto (t_1t_3^{-1},t_2t_3^{\ell}t_4^{-1}).$ Then $\Cl(X_\Sigma)\cong\cA=\Z^2$ and 
$$\mathcal{G}=\ker(\pi)=\{(t_1,t_2,t_1,t_1^{\ell}t_2)\;|\;t_1,t_2\in\K^*\}\cong (\K^*)^2.$$ Hence,  $ T_X \cong {(\K^*)^2}\cong(\K^*)^{4} /\mathcal{G}$ is the torus of $X=X_\Sigma$. The ring $S=\K[x_1,x_2,x_3,x_4]$ is $\Z^2$-graded via  $$\deg_{\cA}(x_1)=\deg_{\cA}(x_3)=(1,0),\quad \deg_{\cA}(x_2)=(0,1), \quad \deg_{\cA}(x_4)=(\ell,1).$$ 
\end{ex}

\begin{ex}
	\label{ex:weighted} The homogeneous coordinate ring of the weighted projective space $X=\Pp(1,w_1,\dots, w_n)$ is $\K[x_0,x_1,\dots,x_n]$ which is $\Z$-graded where $\deg_{\cA} (x_0)=1$ and $\deg_{\cA}(x_i)=w_i>0$ for $i=1,\dots,n$.  
	%Since  $\phi=L_\beta$, $$\phi=\begin{bmatrix} -w_1&-w_2&\cdots&-w_r\\
%	1&0&\cdots&0\\
%	0&1&\cdots&0\\
%	\vdots&\vdots&\ddots&\vdots\\
%	0&0&\cdots&1\\\end{bmatrix}.$$ 
 If $\beta=[1\; w_1\; \cdots\; w_n],$ and $\textbf{u}_1=(-w_1,1,0,\dots,0),\textbf{u}_2=(-w_2,0,1,0,\dots,0),\dots,\textbf{u}_n=(-w_n,0,\dots,0,1)$, then we have the following exact sequences:
$$\dis \xymatrix{ \mathfrak{P}: 0  \ar[r] & \Z^n \ar[r]^{\phi} & \Z^{n+1} \ar[r]^{\beta}& \cA \ar[r]& 0}$$
where  $\phi=[\uu_1\:\uu_2\cdots\uu_n]$ and $L_\beta= \la\textbf{u}_1, \dots, \uu_n \ra$,
$$\dis \xymatrix{ \mathfrak{P}^*: 1  \ar[r] & \mathcal{G} \ar[r]^{i} & (\K^*)^{n+1} \ar[r]^{\pi} & (\K^*)^n \ar[r]& 1}$$
where $\pi:\t\mapsto (t_0^{-w_1}t_1,t_0^{-w_2}t_2,\dots,t_0^{-w_n}t_{n}).$ Then $\Cl(X_\Sigma)\cong\cA=\Z$ and 
$$\mathcal{G}=\ker(\pi)=\{(t,t^{w_1},t^{w_2},\dots,t^{w_n})\;|\;t\in\K^*\}\cong \K^*.$$ Hence,  $ T_X \cong {(\K^*)^n}\cong(\K^*)^{n+1} /\mathcal{G}$ is the weighted  projective torus of $X=X_\Sigma$, where cones of $\Sigma$ are spanned by all proper subsets of the set $\{\vv_1,\dots,\vv_{n+1}\}$, where $\vv_i$ is the $i$-th row of $\phi$ above.	

\end{ex}

\section{Vanishing Ideals via Saturation of Lattice Basis Ideals} \label{S:LatticeBasis}

In this section, we describe the lattice whose ideal is the vanishing ideal $I(Y_Q)$. For any parameterized toric set, we give an algorithm for computing a basis for the unique lattice defining $I(Y_Q)$. This yields a generating set for $I(Y_Q)$ via saturation. 

Recall that $\m=\m^+-\m^-$, where $\m^+\in \N^r$ (respectively, $\m^-\in \N^r$) is the positive (respectively, negative) part of $\m$, and $\x^{\m}$ denotes the monomial $x_1^{m_1}\cdots x_r^{m_r}$ for any $\m=(m_1,\dots,m_r)\in \N^r$. A binomial ideal is an ideal generated by binomials ${x^\textbf{a}}-x^\textbf{b}$, where $\textbf{a},\textbf{b}\in\N^r$, see \cite{Binomialideals} for foundational properties they have. A subgroup $L\subseteq \Z^r$ is called a lattice, and the following binomial ideal  is called the associated lattice ideal:
$$I_L=\la \x^\textbf{a}-\x^\textbf{b}|\textbf{a}-\textbf{b}\in L\ra=\la{\textbf{x}}^{\textbf{m}^+}-\textbf{x}^{\textbf{m}^-}|\textbf{m}\in L\ra.$$
For any matrix $Q$, we denote by $L_Q$ the lattice $\ker_{\Z}Q$ of integer vectors in $\ker Q$.

%\begin{coro}\label{c:binom}
%$I(Y_Q)$ is lattice ideal.
%\end{coro}
%\begin{proof} Let $[p]$ be a point in $X$. Firstly we show that $Y_Q$ is a submonoid of the torus under  the componentwise multiplication operation $[p]\cdot[p']:=[pp']$. If $[{\t}^{\q_{1}}:\cdots:{\t}^{\q_{r}}],[{\t'}^{\q_{1}}:\cdots:{\t'}^{\q_{r}}]\in Y_Q$ then $[{\t}^{\q_{1}}:\cdots:{\t}^{\q_{r}}],[{\t'}^{\q_{1}}:\cdots:{\t'}^{\q_{r}}]=[{(\t\t')}^{\q_{1}}:\cdots:{(\t\t')}^{\q_{r}}]\in Y_Q.$ And furthermore if we take $t_i=1$ for any $i$, then $[{\t}^{\q_{1}}:\cdots:{\t}^{\q_{r}}]=[1:\cdots:1]\in Y_Q$. Thus  $Y_Q$ is a submonoid and so, $I(Y_Q)$ is lattice ideal from \cite{Sahin}.
%\end{proof}

\begin{lm}\label{lm:hom} A binomial $f=\textbf {x}^{\textbf{a}}-\textbf {x}^{\textbf{b}}$ in $S$ is homogeneous iff $\textbf{a}-\textbf{b}\in L_\beta$.
\end{lm}
\begin{proof} By definition, $\deg_{\cA}(\textbf {x}^{\textbf{a}})={a_1}\deg_{\cA}(x_1)+\cdots+{a_r}\deg_{\cA}(x_r)=\beta_1{a_1}+\cdots+\beta_r{a_r}= \beta({\textbf{a}})$. So, $f$ is homogeneous, that is, $\deg_{\cA}(\textbf {x}^{\textbf{a}})=\deg_{\cA}(\textbf {x}^\textbf{b})$ iff $\beta({\textbf{a}})=\beta({\textbf{b}})$. The latter is equivalent to $\beta({\textbf{a}}-{\textbf{b}})=0$, which holds true iff
$\textbf{a}-\textbf{b}\in  L_\beta$.
\end{proof}

 The fact that $I(Y_Q)$ is a lattice ideal has recently been observed in \cite{Sahin} without describing the corresponding lattice. It is now time to describe the missing lattice.

\begin{lm}\label{lm:equality} The ideal $I(Y_Q)$ is equal to  the lattice ideal $I_{L}$ for $L=\{\m\in L_\beta: Q\m\equiv 0\;\mbox{mod}\;(q-1)\}$.
\end{lm}	
\begin{proof} Before we go further, let us note that $x^\textbf{a}(\t^{\q_1},\dots,\t^{\q_r})=(\t^{\q_1})^{a_{1}}\cdots (\t^{\q_r})^{a_{r}}={\t}^{Q\textbf{a}}$, for $\t\in({\K}^*)^s$. It follows that a binomial $f=x^{\textbf{a}}-x^{\textbf{b}}$ vanishes at a point $(\t^{\q_1},\dots,\t^{\q_r})$ if and only if ${\t}^{Q\textbf{a}}={\t}^{Q\textbf{b}}$. As $\t\in({\K}^*)^s$, this is equivalent to $\t^{Q(\textbf{a}-\textbf{b})}=1$.

To prove $I(Y_Q)\subseteq I_{L}$, take a generator  $f=x^{\textbf{a}}-x^{\textbf{b}}$ of $I(Y_Q)$. As $f$ vanishes on $Y_Q$, we have that $\t^{Q(\textbf{a}-\textbf{b})}=1$ for all $\t\in({\K}^*)^s$. Then, by substituting $\t=(\eta,1,\dots,1)$  in this equality, we observe that $q-1$ divides the first entry of the row matrix $Q(\textbf{a}-\textbf{b})$, where $\eta$ is a generator of the cyclic group $\K^*$ of order $q-1$. Similarly, $q-1$ divides the other entries, and so $Q(\textbf{a}-\textbf{b})\equiv 0\;\mbox{mod}\;(q-1)$. Since $\textbf{a}-\textbf{b}\in L_\beta$ from Lemma \ref{lm:hom}, $f$ being homogeneous, we have $\textbf{a}-\textbf{b}\in {L}$.

Conversely, let $f=x^{\textbf{a}}-x^{\textbf{b}}\in I_L$. Then $\textbf{a}-\textbf{b}\in L_\beta$ and $Q(\textbf{a}-\textbf{b})\equiv 0\;\mbox{mod}\;(q-1)$. This implies that $f$ is homogeneous by Lemma \ref{lm:hom} and that $\t^{Q(\textbf{a}-\textbf{b})}=1$ for all $\t\in({\K}^*)^s$. Hence, $f(\t^{\q_1},\dots,\t^{\q_r})=0$ for any $\t\in({\K}^*)^s$, by the first part. Thus, $f\in I(Y_Q)$ and $I_L\subseteq I(Y_Q)$.
\end{proof}	

For any lattice $L$, the lattice basis ideal of $L$ is the ideal of $S$ generated by the binomials $\textbf{x}^{\textbf{m}^+}-\textbf{x}^{\textbf{m}^-}$ corresponding to the vectors $\textbf{m}$ which constitute a $\Z$ - basis of $L$.

Let $I$ and $J$ be ideals in $S$. Then the ideal 
$$I:J^{\infty}=\{F\in S \ : \ F \cdot J^k \subseteq I \quad \mbox{for some integer} \;\; k\geq 0\}$$
is called the saturation of $I$ with respect to $J$.

\begin{lm}\label{l:latticebasisideal}\cite[Lemma 7.6]{CombinatorialCommutativeAlgebra} Let $L$ be a lattice. The saturation of the lattice basis ideal of $L$ with respect to the ideal $\la x_1\cdots x_r\ra$ is equal to the lattice ideal $I_L$.
\end{lm}

Thus, we can obtain generators of $I(Y_Q)=I_{L}$ from a $\Z$-basis of $L$. Although the lattice $L$ in Lemma \ref{lm:equality} is inevitable conceptually, finding its basis is a difficult task in general. The following result gives another description of $L$ leading to an algorithm computing its basis.

\begin{tm}\label{t:mod} Let $\pi_s:{\Z}^{n+s}\to{\Z}^{n}$ be the projection map
sending $(c_1,\dots,c_n,c_{n+1},\dots,c_{n+s})$ to $(c_1,\dots,c_n)$.  Then $I(Y_Q)=I_L$, for the lattice $L=\{\phi\textbf{c}: \textbf{c}\in\pi_s\left(\mbox{ker}_{\Z}[Q\phi|(q-1)I_s]\right)\}$. Furthermore, columns of the matrix $\phi M$ constitute a basis for $L$, where $M$ is a matrix whose columns are the first $n$ coordinates of the generators of 
$\mbox{ker}_{\Z}[Q\phi|(q-1)I_s]$.
\end{tm}	
\begin{proof} We have that $I(Y_Q)=I_{L_1}$ where $L_1=\{\m\in L_\beta: Q\m\equiv 0\;\mbox{ mod}\; (q-1)\}$ by Lemma \ref{lm:equality}. Therefore it is enough to prove that $L=L_1$. Since $\mbox{Im}\phi=L_\beta$ by the exact sequence $ \mathfrak{P}$, it follows that $\m\in L_\beta$ iff $\m=\phi\textbf{c}$ for some $\textbf{c}\in \Z^n$. This means that
$$L_1=\{\phi\textbf{c}:Q \phi\textbf{c}\equiv 0\;\mbox{mod}\; (q-1)\;\mbox{and} ~~\textbf{c}\in {\Z}^n\}.$$
Take  $\phi \textbf{c} \in L $ so that $\textbf{c}=(c_1,\dots,c_n)\in \pi_s\left(\mbox{ker}_{\Z}[Q\phi|(q-1)I_s]\right)$. Then there are $c_{n+1},\dots,c_{n+s} \in \Z$ such that $[Q\phi|(q-1)I_s](c_1,\dots,c_n,c_{n+1},\dots,c_{n+s})=0$. This is equivalent to 
$$Q\phi(c_1,\dots,c_n)+(q-1)I_s(c_{n+1},\dots,c_{n+s})=0, \; \mbox{or}$$
$$Q\phi\textbf{c}=-(q-1)(c_{n+1},\dots,c_{n+s}).$$
This proves that $Q \phi\textbf{c}\equiv 0\;\mbox{mod}\;(q-1)$. Thus $\phi \textbf{c}\in L_1.$

For the converse, take $\phi\textbf{c}\in L_1$. Then $Q\phi \textbf{c} \equiv 0\;\mbox{mod}\;(q-1)$. It follows that
$$Q\phi\textbf{c}=(q-1)(c_{n+1},\dots,c_{n+s})$$
for some $c_{n+1},\dots,c_{n+s}\in {\Z}$. Thus, we have $[Q\phi|(q-1)I_s](c_1,\dots,c_n,-c_{n+1},\dots,-c_{n+s})=0.$
Hence, we have $\textbf{c}=\pi_s(c_1,\dots,c_n,-c_{n+1},\dots,-c_{n+s})\in \pi_s\left(\mbox{ker}_{\Z}[Q\phi|(q-1)I_s]\right).$ Thus, $\phi\textbf{c}\in L$, completing the proof of the claim that $I(Y_Q)=I_L$.

We next prove that the columns of $\phi M$ form a basis for $L$, where $M$ is a matrix whose columns are the first $n$ coordinates of the generators of 
$\mbox{ker}_{\Z}[Q\phi|(q-1)I_s]$. As the matrix $B=[Q\phi|(q-1)I_s]$ has rank $s$, its kernel $\mbox{ker}_{\Z}B$ has rank $n$. If $A$ is the matrix whose columns $A_1,\dots, A_n$ form a basis for $\mbox{ker}_{\Z}B$, then $\mbox{im} (A)=\mbox{ker}_{\Z}B$ and that $M=[I_n|0_{n\times s}]A$. Take  any element $\phi\textbf{c}\in L $, where $\textbf{c}\in\pi_s\left(\mbox{ker}_{\Z}B\right)$. Thus, we can write $\textbf{c}={A_1}k_1+\cdots+{A_n}k_n=A[k_1\cdots k_n]$ for some $k_1,\dots, k_n\in\Z$ yielding $\phi\textbf{c}=\phi M [k_1\cdots k_n]$. Therefore $L$ is spanned by $n=\rank L$ columns of the $r\times n$ matrix $\phi M$. Hence, these columns constitute a basis for $L$.
\end{proof}	

Theorem \ref{t:mod} leads to the following algorithm for computing a $\Z$-basis of the lattice $L$ for which we have $I(Y_Q)=I_L$. 

\begin{algorithm}
\caption{ Computing a basis for the lattice $L$ such that  $I_L=I(Y_Q)$.}\label{a:lattice1}
\begin{flushleft}
\hspace*{\algorithmicindent} \textbf{Input} The matrices $Q\in M_{s\times r}(\Z)$, $\phi\in M_{r\times n}(\Z)$ and a prime power $q$.\\  
\hspace*{\algorithmicindent} \textbf{Output} A basis of $L$.
\end{flushleft}
\begin{algorithmic}[1]
\State Find the generators of the lattice $\mbox{ker}_{\Z}[Q\phi|(q-1)I_s]$.
\State Find the matrix $M$ whose columns are the first $n$ coordinates of the generators of $\mbox{ker}_{\Z}[Q\phi|(q-1)I_s]$.
\State Compute  the matrix $\phi M$ whose columns are a $\Z$-basis of the lattice $L$
\end{algorithmic}
\end{algorithm}

The algorithm can be implemented in \verb|Macaulay2| as follows.

\begin{pdr}\label{pdr:lattice1} The command \verb|ML| gives the matrix whose columns are generators of the lattice $L$. 
{\scriptsize
\begin{verbatim}
i2: s=numRows Q;n=numColumns Phi;
i3: ML=Phi*(id_(ZZ^n)|(random(ZZ^n,ZZ^s))*0)*(syz (Q*Phi|(q-1)*(id_(ZZ^s))))
\end{verbatim}}
\end{pdr}
\begin{pdr}\label{pdr:saturation} A generating set for $I(Y_Q)$ via saturation.
{\scriptsize\begin{verbatim}
i4: r=numRows Phi; (D,P,K) = smithNormalForm Phi; Beta=P^{n..r-1};
i5: S=ZZ/q[x_1..x_r,Degrees=>transpose entries Beta];
i6: toBinomial = (b,S) -> (top := 1_S; bottom := 1_S;
    scan(#b, i -> if b_i > 0 then top = top * S_i^(b_i)
    else if b_i < 0 then bottom = bottom * S_i^(-b_i)); top - bottom);
i7: IdealYQ=(ML,S)->(J = ideal apply(entries transpose ML, b -> toBinomial(b,S)); 
    scan(gens S, f-> J=saturate(J,f));J)
i8:IYQ=IdealYQ(ML,S)    
\end{verbatim}}
\end{pdr}
\begin{ex} \label{ciex}
Let $X=\cl H_{2}$ over $\F_{11}$ and $Q=[1\;  2\;  3\;  4 ]$. So, we have the following input:
{\scriptsize\begin{verbatim}
i1 : q=11;Phi=matrix{{1,0},{0,1},{-1,2},{0,-1}}; Q=matrix {{1,2,3,4}};
\end{verbatim}}	
\noindent Procedure \ref{pdr:lattice1} gives the following matrix whose columns constitute a basis of $L$:
 $$\verb|ML|={\begin{bmatrix}~~2 & ~~1 & ~~0&-1\\-5 & ~~0 & ~~5 & ~0\\\end{bmatrix}}^T.$$ 
Finally, we determine $I(Y_Q)=I_L$ using Procedure \ref{pdr:saturation} and get $I_L=\la x_1^2x_2-x_4, x_1^5-x_3^5 \ra$.
\end{ex}

\begin{rema} \label{rem:CI}Another advantage of finding the matrix \verb|ML| giving a basis for the lattice is that one can confirm if the lattice ideal is a complete intersection immediately, by checking if \verb|ML| is mixed dominating.
\end{rema}

\begin{defi} Let $A$ be matrix whose entries are all integers. $A$ is called mixed if there is a positive and a negative entry in every column. If no square submatrix of $A$ is mixed, it is called dominating.
\end{defi} 

\begin{tm}\label{t:mixeddominating}\cite[Theorem 3.9]{MT} Let $L\subseteq\Z^r$ be a lattice with the property that $L\cap \N^r={0}$. Then, $I_L$ is complete intersection $\iff L$ has a basis ${\m_1,\dots,\m_k}$ such that the matrix $[\m_1\cdots \m_k]$ is mixed dominating. In the affirmative case, we have 
$$ I_L=\la \textbf{x}^{\textbf{m}_1^+}-\textbf{x}^{\textbf{m}_1^-},\dots, \textbf{x}^{\textbf{m}_k^+}-\textbf{x}^{\textbf{m}_k^-} \ra. $$
\end{tm}

Using Theorem \ref{t:mixeddominating}, one can confirm when $I(Y_Q)=I_L$ is a complete intersection by looking at a basis of the lattice $L$.

\begin{ex} \label{Ex:CI}
Let $X=\cl H_{\ell}$ be the Hirzebruch surface over $\F_{q}$, where $q$ is odd. For any positive integers $q_1$ and $q_2$, consider $Q=[q_1\;  q_2\;  q_1+2\;  \ell q_1+q_2]$. We will compute generators of $I(Y_Q)$ for all $q$ at once using Lemma \ref{lm:equality}.  The key observation is that $Y_Q=Y_{Q'}$ for the matrix $Q'=[0\;  0\;  2\;  0]$, because we have $[t^{q_1}:t^{q_2}:t^{q_1+2}:t^{\ell q_1+q_2}]=[1:1:t^2:1]$ in $X$ for all $t\in\K^*$ from Example \ref{ex:Hirzebruch}. Recall that the ideal $ \mathrm{I}(Y_{Q'})=\mathrm{I}_{L}$ for the lattice described by $ L=\{\m\in L_{\beta} \quad | \quad  Q'\cdot \m \equiv 0 \quad \mbox{mod}(q-1)\}.$
Since $L_{\beta}$ is spanned by the columns $\uu_1$ and $\uu_2$ of the matrix $$\phi=\begin{bmatrix}
1 & 0 & -1& ~~0 \\
0 & 1 & ~~~\ell& -1
\end{bmatrix}^T,$$ 
it follows that $\m\in L_{\beta}$ if and only if $\m=\uu_1a_1+\uu_2a_2=(a_1,a_2,-a_1+\ell a_2,-a_2),$ for some $a_1,a_2\in\Z$. Thus, we obtain $Q'\cdot \m=-2a_1+2\ell a_2$ for $\m\in L_{\beta}$. Therefore, $\m\in L \iff -2a_1+2\ell a_2=(q-1)k$, for some $k\in \Z$. Since $q-1$ is even, the last condition is equivalent to $a_1=\ell a_2-k\frac{q-1}{2}$ in which case $\m=a_2\m_1-k\m_2$ , where $\m_1=\ell \uu_1+\uu_2$ and $\m_2=\left(\frac{q-1}{2}\right)\uu_1$. Therefore, the matrix $\verb|ML|$ whose columns $\m_1$ and $\m_2$ constitute a basis of $L$, is given by:   
$$\verb|ML|=\begin{bmatrix}~\ell & 1 & ~~~0  & ~-1\\~(q-1)/2 & ~~~0 & -(q-1)/2&~~~0
\end{bmatrix}^T.$$ 
Since $\verb|ML|$ is mixed dominating, it follows that $I(Y_Q)=\mathrm{I}(Y_{Q'})=I_L$ is a complete intersection. Therefore, without the saturation Procedure \ref{pdr:saturation}, we get  $I(Y_Q)=I_L=\la x_1^{\ell}x_2-x_4,x_1^{(q-1)/2}-x_3^{(q-1)/2}\ra$ immediately. Notice that by taking $q=11$, $q_1=1$, $q_2=2$ and $\ell=2$, we recover the Example \ref{ciex}.

%In contrast to that example, as in Example \ref{ex:gb}. 
\end{ex}

\section{ The order of the subgroup $Y_Q$}\label{S:length}

In this section, we give an algorithm computing the size of $Y_Q$, which is the length of the parameterized toric code ${\cC}_{\aa,Y_Q}$, directly using the parameterization of $Y_Q$. The order of this subgroup can also be computed using the vanishing ideal of $Y_Q$ only if an element of $\reg(Y_Q)$ is known, which is a difficult task to achieve. When $Y_Q$ is a complete intersection of hypersurfaces of degrees $\aa_1,\dots,\aa_n$, it is shown that $\aa_1+\cdots+\aa_n \in \reg(Y_Q)$, so that the order is $H_{Y_Q}(\aa_1+\cdots+\aa_n)$, see \cite[Theorem 3.6]{sasop}. However, if $Y_Q$ is not a complete intersection, no specific element of $\reg(Y_Q)$ is known. In these cases, it is natural to use the size $|Y_Q|$ in order to determine $\reg(Y_Q)$, see Example \ref{P(2,2,3,5)}.

It is clear that $T_X$ and $Y_Q$ are groups under the componentwise multiplication $$[p_1:\cdots:p_r][p'_1:\cdots :p'_r]=[p_1p'_1:\cdots: p_rp'_r]$$
and that the map 
$$\varphi_Q : (\K^{*})^{s}\to Y_Q,  \quad  \t\to[\t^{\q_1}:\cdots:\t^{\q_r}] $$ 

\noindent is a group epimorphism. It follows that $Y_Q\cong (\K^{*})^{s}/ \mbox{ker}(\varphi_Q)$ and so,
$$|Y_Q|=|(\K^{*})^{s}|/|\mbox{ker}(\varphi_Q)|=(q-1)^s/|\mbox{ker}(\varphi_Q)|.$$ Hence, the length of the code ${\cC}_{\aa,Y_Q}$ depends on $|\mbox{ker}(\varphi_Q)|$.

Let $\square_q=[0,q-2]^s$ be the hypercube inside $\R^s$ determined by the field $\K=\F_q$ and $\eta$ be a generator of $\K^*$. Let $\textbf{H}=\{\hh \in \square_q \cap \Z^s \quad|\quad \hh Q\phi\equiv 0\;\mbox{mod}\;q-1\}$. We first prove that there is a one to one correspondence between the kernel $\mbox{ker}(\varphi_Q)$ and the set $\textbf{H}$.
\begin{pro}\label{P:length1}  We have $\mbox{ker}(\varphi_Q) =\{(\eta^{h_1},\dots,\eta^{h_s})|\hh=(h_1,\dots,h_s)\in \textbf{H}\}$ and thus
$|\mbox{ker}(\varphi_Q)|=|\textbf{H}|$.
\end{pro}
\begin{proof} Let $\t \in \mbox{ker}(\varphi_Q)\subseteq (\K^*)^s$ . Then $[{\t}^{\q_{1}}:\cdots:{\t}^{\q_{r}}]=[1:\cdots:1]$, that is,$({\t}^{\q_{1}},\dots,{\t}^{\q_{r}})$ is element of the orbit $\mathcal{G}(1,\dots,1)=\mathcal{G}=\{\x \in (\K^*)^r \:|\:  \x^{\m}=1 \;\mbox{for all}\;\m\in  L_\beta\}$. Since $L_\beta=\phi(\Z^n)$, we have $\m=\phi \textbf{c}\in L_{\bb}$ for any $\textbf{c}\in \Z^n$ and thus
\begin{equation}\label{e:G}
\x^{\m}(\t^{\q_1},\dots,{\t}^{\q_r})=\t^{Q \m}=\t^{Q \phi \textbf{c}}=1, \quad \mbox{for any}\quad \textbf{c}\in \Z^n.
\end{equation} 
Since every $\t$ in $(\K^*)^s$ satisfies $\t =(\eta^{h_1},\dots,\eta^{h_s})$  for some $\hh=(h_1,\dots,h_s)\in \square_q \cap \Z^s$, the equality (\ref{e:G}) implies that $\eta^{\hh Q \phi \textbf{c}}=1$, for all $\textbf{c}\in \Z^n$. Thus, $\hh Q \phi \textbf{c} \equiv 0\;\mbox{mod}\;q-1$ for all $\textbf{c}\in \Z^n$. By choosing $\textbf{c}=\textbf{e}_i$, for all $i=1,\dots,n$, where $\textbf{e}_i$ is a standard basis vector of $\Z^n$,  we observe that $\hh Q \phi \equiv 0\;\mbox{mod}\;q-1$. This implies that $\mbox{ker}(\varphi_Q)\subseteq \{(\eta^{h_1},\dots,\eta^{h_s})|\hh\in \textbf{H}\}$. The other inclusion is straightforward, completing the first part of the proof. Since the order of $\eta$ is $q-1$ and the integers $h_i$ lie in $[0,q-2]$, it is clear that the correspondence between $\mbox{ker}(\varphi_Q)$ and $\textbf{H}$ is one to one.
\end{proof}

\begin{pdr}\label{pdr:length} Given matrices $Q$ and $\phi$, and the prime power $q$, the following \verb|Macaulay2|  procedure allows one to compute the length of ${\cC}_{\aa,Y_Q}$. The list \verb|A| in the fifth step consists of the elements in $\square_q \cap \Z^s$. In the sixth step, we check whether the elements of $\square_q \cap \Z^s$ is in the set $\textbf{H}$ or not and compute $k=|\textbf{H}|$. 
{\scriptsize\begin{verbatim}
i2 : r=numRows Phi;s=numRows Q;n=numColumns Phi;k=0;
i3 : L=for i from 1 to q-1 list i;
i4 : L= set L;L=L^**(s);L=toList L;
i5 : A= apply(L,i->toList deepSplice i)
i6 : scan(A,i-> if ((matrix{i}*Q*Phi)%(map((ZZ)^1,n,(i,j)->(q-1))))
==(matrix mutableMatrix(ZZ,1,n)) then k=k+1);
i7 : N=((q-1)^s)/k
\end{verbatim}}
\end{pdr}

\begin{ex} Let $X=\cl H_{2}$ over $\F_{11}$ and $Q=[1\;  2\;  3\;  4 ]$ as in Example \ref{ciex}. Let us calculate the length of the codes arising from $Q$ using the Hilbert function of the vanishing ideal $I_L=\la x_1^2x_2-x_4, x_1^5-x_3^5 \ra$ found there. As the degrees of the variables are $$\deg(x_1)=\deg(x_3)=\bb_1=\bb_3=(1,0), \quad \deg(x_2)=\bb_2=(0,1),\quad \deg(x_4)=\bb_4=(2,1),$$ the degrees of the generators are $\alpha_1=(2,1)$ and $\alpha_2=(5,0)$. By \cite[Theorem 3.1]{sasop}, we can assure that the element $\aa_1+\aa_2=(7,1) \in \reg(Y_Q)$, so that the size is the value of the Hilbert function at $(7,1)$.  Thus, we compute $|Y_Q|=5$ by the command below, right after computing $I(Y_Q)$ using the Procedure \ref{pdr:saturation}:
{\scriptsize\begin{verbatim}
i9 : hilbertFunction({7,1},IYQ)
\end{verbatim}}
The same length can be computed directly using the Procedure \ref{pdr:length} with the following input:
 {\scriptsize\begin{verbatim}
i1 : q=11;Phi=matrix{{1,0},{0,1},{-1,2},{0,-1}}; Q=matrix {{1,2,3,4}};
\end{verbatim}}
 \end{ex}
 
The following example is to illustrate the advantage of computing length beforehand in order to determine $\reg(Y_Q)$ and to obtain a finite list of interesting codes. Notice that the order of $Y_Q$ cannot be computed as in the previous example using \cite[Theorem 3.1]{sasop} since the vanishing ideal is not a complete intersection.
  \begin{ex}\label{P(2,2,3,5)} Fix $q=5$ and consider the incidence matrix $Q$ of the square shaped graph with vertices $V=\{1,2,3,4\}$ and edges $E=\{\{1,2\},\{2,3\},\{3,4\},\{1,4\}\}$. We first compute a minimal generating set for the vanishing ideal of the subgroup $Y_Q \subseteq \Pp(2,2,3,5)$:
$$x_1^4-x_2^4, ~~~x_1^2x_2^6-x_3^2x_4^2, ~~~ x_3^6-x_1^2x_2^2x_4^2, ~~~ x_2^4x_3^4-x_4^4. $$
It follows that $I(Y_Q)$ is not a complete intersection. So, we cannot compute the order $|Y_Q|$ as before. We calculate $|Y_Q|$ directly using the Procedure \ref{pdr:length} with the following input:
\begin{verbatim}
i1 : q=5; Phi= syz matrix {2,2,3,5}; 
Q=matrix{{1,0,0,1},{1,1,0,0},{0,1,1,0},{0,0,1,1}};
\end{verbatim}
This reveals that $|Y_Q|=32$. Hence, $Y_Q$ contains half of the points inside the torus $T_X$. 
 
Since $|Y_Q|=32$, $\reg(Y_Q)=\{i \in \N\bb \: : \: H_{Y_Q}(i)=32\}.$ Using the conditional non-decreasing behavior of the Hilbert function noted in \cite{sasop}, we see that $H_{Y_Q}(i)\leq H_{Y_Q}(i+w)\leq 32$, for $w\in \{2,3,5\}$ and for all $i>0$. This means that if $H_{Y_Q}(i)=32$ for some $i=i_0$ then $H_{Y_Q}(i_0+2)=H_{Y_Q}(i_0+3)=32$ and thus $H_{Y_Q}(i_0+j)=32$ for all $j>3$. Thus, we just need to determine $i_0$ with this property and $H_{Y_Q}(i_0+1)$. The following command finds these values:
\begin{verbatim}
for i from 0 to 100
do (
    if hilbertFunction(i,IYQ)<32 then print hilbertFunction(i,IYQ)
    else stop 
    and print [i,hilbertFunction(i,IYQ),hilbertFunction(i+1,IYQ)]  
    );
\end{verbatim}

The output is ${1,0,2,1,3,3,5,5,7,8,10,11,14,14,18,19,21,24,24,28,27,31,29,
[23, 32, 31]}$. Here, $[23, 32, 31]$ means that $i_0=23$, $H_{Y_Q}(i_0)=32$ and $H_{Y_Q}(i_0+1)=31$, and hence we determine the regularity as $\reg(Y_Q)=\{23\}\cup \{25+\N\}$. So, it suffices to consider the codes corresponding to $\aa$ in the set $$\N\bb\setminus \reg(Y_Q)=\{2,3,4,5,6,7,8,9,10,11,12,13,14,15,16,17,18,19,20,21,22\}\cup\{24\}.$$ Therefore, we obtain a finite list of interesting codes along with two of their parameters; the length is $32$ and the dimensions are found above as the values of the Hilbert function of $I(Y_Q)$.
\end{ex}
 
  We conclude the section with a polyhedral method to compute the size of the set $Y_Q$. We next prove that the elements of the kernel $\mbox{ker}(\varphi_Q)$ correspond bijectively to lattice points lying inside the polytope $\cl P=\{(\hh,\kk)\in  \R^s \times \R^n \quad|\quad \hh Q\phi=(q-1)\kk \quad \mbox{and} \quad  \hh \in \square_q \}$.
\begin{pro} \label{P:length2} We have $|\mbox{ker}(\varphi_Q)|=|\cl P \cap \Z^{s+n}|$.
\end{pro}
\begin{proof} By Proposition \ref{P:length1}, there is a one to one correspondence between  $\mbox{ker}(\varphi_Q)$ and $\textbf{H}$. Hence, it suffices to show that there is a bijection between $\textbf{H}$ and $\cl P \cap \Z^{s+n}$. If $\hh \in \textbf{H}$, then  $\hh Q \phi \equiv 0\;\mbox{mod}\;q-1$. Thus, there is some $\kk=({k_1},\dots,{k_n})\in \Z^n$ such that $\hh Q \phi=(q-1)\kk$. Therefore, $(\hh,\kk)\in\cl P \cap \Z^{s+n}$. Notice that there is exactly one such $\kk$ for every $\hh$, as $(q-1)\kk=(q-1)\kk'$ implies $\kk=\kk'$. Conversely, if
$(\hh,\kk)\in\cl P \cap \Z^{s+n}$, then $\hh Q \phi=(q-1)\kk$ so that $\hh Q \phi \equiv 0\;\mbox{mod}\;q-1$, completing the proof.
\end{proof}

\begin{rema} \label{R:polytope} When $X=\Pp^{r-1}$ is the $n=r-1$ dimensional projective space, \cite[Proposition 3.3]{Algebraicmethodsforparameterizedcodesandinvariantsofvanishingoverfinitefields} gives a bijection between the set $\mbox{ker}(\varphi_Q)$ and the lattice points in the polytope $\textbf{P}$ so that $|\mbox{ker}(\varphi_Q)|=|\textbf{P} \cap \Z^{s+r+1}|$, where 
$$\textbf{P}=\{(\hh,\lambda,\mu) \in \R^s \times \R^r \times \R \quad |\quad \hh Q=(q-1)\lambda+ \mu \textbf{1}, \quad \hh\in \square_q \quad \mbox{and} \quad  0\leq \mu \leq q-2\} \quad \mbox{and} \quad \textbf{1}\in \Z^r.$$
Even in this special case, toric point of view improves upon \cite[Proposition 3.3]{Algebraicmethodsforparameterizedcodesandinvariantsofvanishingoverfinitefields} in the sense that our polytope $\cl P$ lies in $\R^{s+n}=\R^{s+r-1}$ whereas $\textbf{P}$ lies in $\R^{s+r+1}$, which increases the complexity of computing the lattice points.
\end{rema}
\begin{ex}\label{E:polytope} Let us revisit \cite[Example 3.4]{Algebraicmethodsforparameterizedcodesandinvariantsofvanishingoverfinitefields}. So, $X=\Pp^3$ over $\F_5$, $\phi$ is the matrix with columns $(-1,1,0,0)$, $(-1,0,1,0)$ and $(-1,0,0,1)$. Consider the incidence matrix $Q$ of the square shaped graph with vertices $V=\{1,2,3,4\}$ and edges $E=\{\{1,2\},\{2,3\},\{3,4\},\{1,4\}\}$. So, $Q$ is the matrix with columns $(1,1,0,0)$, $(0,1,1,0)$, $(0,0,1,1)$ and $(1,0,0,1)$. 
 Using Sage \cite{sage} we compute in $0,04 $ seconds the following $16$ integral points $(h_1,h_2,h_3,h_4,k_1,k_2,k_3)$ of the $4$ dimensional compact polytope $\cl P\subset \R^7$ which is the convex hull of $16$ vertices.
{\scriptsize\begin{verbatim}
 (0, 0, 0, 0, 0, 0, 0), (0, 1, 0, 1, 0, 0, 0), (0, 2, 0, 2, 0, 0, 0), (0, 3, 0, 3, 0, 0, 0),
 (1, 0, 1, 0, 0, 0, 0), (1, 1, 1, 1, 0, 0, 0), (1, 2, 1, 2, 0, 0, 0), (1, 3, 1, 3, 0, 0, 0),
 (2, 0, 2, 0, 0, 0, 0), (2, 1, 2, 1, 0, 0, 0), (2, 2, 2, 2, 0, 0, 0), (2, 3, 2, 3, 0, 0, 0),
 (3, 0, 3, 0, 0, 0, 0), (3, 1, 3, 1, 0, 0, 0), (3, 2, 3, 2, 0, 0, 0), (3, 3, 3, 3, 0, 0, 0).
\end{verbatim}}
Therefore, the $16$ points in the subgroup $Y_Q$ are found to be $(\eta^{h_1},\eta^{h_2},\eta^{h_3},\eta^{h_4})$ for $(h_1,h_2,h_3,h_4)$ that appeared above. We also compute in $1$ second the $16$ lattice points of the $5$ dimensional polytope $\textbf{P}\subset \R^9$ which is the convex hull of $32$ vertices.
\end{ex}

\section{Parameterized Toric Codes ${\cC}_{\aa,Y_Q}$}\label{S:minimumdistance}	

In this section, we apply algebraico-geometric techniques developed in previous sections to evaluation codes on subgroups $Y_Q$. Recall that
these linear codes are images of the following evaluation map
$$\ev_{Y_Q}:S_\aa\to \K^N,\quad F\mapsto (F(P_1),\dots,F(P_N)).$$ 
The code ${\cC}_{\aa,Y_Q}:=\text{ev}_{Y_Q}(S_\aa) \subseteq \F_q^N$ is called the \textbf{parameterized toric code} associated to $Q$. There are $3$ main parameters $[N,K,\delta]$ of a linear code. The \textit{length} $N$ of ${\cC}_{\aa,Y_Q}$ is the order $|Y_Q|$ of the subgroup in our case studied in Section \ref{S:length}. The \textit{dimension} of ${\cC}_{\aa,Y_Q}$, denoted $K=\dim_{\K}({\cC}_{\aa,Y_Q})$, is the dimension of the image as a subspace of $\F_q^N$. The number of non-zero entries in any $c\in{{\cC}_{\aa,Y_Q}}$ is called its \textit{weight} and  \textit{minimum distance} $\delta$ of ${\cC}_{\aa,Y_Q}$ is the smallest weight among all codewords $c\in{{\cC}_{\aa,Y_Q}}\setminus\{0\}$. 	These parameters are related by the Singleton's bound given by $\delta\leq N+1-K$. A code is called MDS (maximum distance separable), if $\delta$ attains its maximum value, i.e. $\delta=N+1-K$.

Recall from Section \ref{S:length} that $\varphi_Q : (\K^{*})^{s}\to Y_Q,  \quad  \t\to[\t^{\q_1}:\cdots:\t^{\q_r}] $ and $|\textbf{H}|=|\mbox{ker}(\varphi_Q)|$ so that the length of the code is $|Y_Q|=(q-1)^s/|\textbf{H}|$. This map will also be used to give a lower bound on the minimum distance as we discuss now. A key observation is that the composition $F\circ\varphi_Q  $ defines a map $ (\K^{*})^{s}\to \K$, for a polynomial $F\in S_{\aa}$. Thus, $(F\circ\varphi_Q)(t_1,\dots,t_s)=F(\t^{\q_1}:\cdots:\t^{\q_r})$, for any $(t_1,\dots,t_s)\in (\K^{*})^{s}$. 

As the weight of the codeword $ev_{Y_Q}(F)$ is determined by the number of zeros of $F$ inside $Y_Q$, the idea is to compute this number by counting zeros of $F\circ\varphi_Q$ inside $(\K^{*})^{s}$. The following result will be used for this purpose.

\begin{lm}\label{L:numberzero}\cite[Lemma 3.2]{MinDisProjTorusVillarreal} Let $ G(y_1,y_2,\dots,y_s)$ be a non-zero polynomial over $\K=\F_q$ of total degree $d$. Then, the number of zeros of $G$ in $(\K^*)^s$ satisfies $|V(G)\cap(\K^*)^s|\leq d(q-1)^{s-1}$. $\hfill \Box$
\end{lm}

Let $\textbf{x}^{\textbf {a}}=x_1^{a_1}\cdots x_r^{a_r}\in S_\alpha$.  Substituting $x_i=\textbf {y}^{\q_{i}}$ in $\textbf{x}^{\textbf {a}}$ yields the monomial $\textbf{y}^{Q\textbf{a}}=y_1^{Q_1\textbf{a}}\cdots y_s^{Q_s\textbf{a}}$ and so $\mbox{deg}_{y_i}(\textbf{y}^{Q\textbf{a}})=Q_i \textbf{a}$, where $Q_i$ is the $i-$th row of $Q$. Let $\overline{Q_i \textbf{a}}$ be remainder of $Q_i \textbf{a}$ upon division by $q-1$. Then the following number will be crucial in our lower bound:
 $$d(\alpha,Q)=\mbox{max}\{\overline{Q_1\textbf{a}}+\cdots+\overline{Q_s\textbf{a}}\quad |\quad \textbf{x}^{\textbf {a}}\in S_{\aa}\}.$$

\begin{tm}\label{t:mindislower} The minimum distance of the code ${\cC}_{\aa,Y_Q}$ satisfies $$\delta({\cC}_{\aa,Y_Q})\geq\frac{(q-1)^{s-1}}{|\textbf{H}|}[q-1-d(\alpha,Q)].$$
\end{tm}
\begin{proof} Let $c=ev_{Y_Q}(F)$ be the codeword corresponding to the homogeneous polynomial $F\in S_{\aa}$. Then, its weight is by definition the number of non-zero components, which is the difference between the number of total components and the number of zeros of $F$ on $Y_Q$: $|Y_Q|-|V_X(F)\cap Y_Q|$. Therefore, the minimum of the weights corresponding to nonzero codewords is given by
	$$\delta({\cC}_{\aa,Y_Q})=|Y_Q|-\mbox{max}\{|V_X(F)\cap Y_Q|\: : \:F\in S_\alpha\setminus I_\alpha(Y_Q)\}.$$
	For a homogeneous polynomial $F\in S_\alpha$, we have $F\in I(Y_Q) \iff F \circ \varphi_Q \in I((\K^*)^s)$. Since $I((\K^*)^s)$ is generated by the binomials $\{y_1^{q-1}-1,\dots,y_s^{q-1}-1\}$, if $F\notin I(Y_Q)$ and $G$ is the remainder of $F(\textbf{y}^{\q_1},\dots,\textbf{y}^{\q_r})$ in $\K[y_1,\dots,y_s]$ under division
	by the set $\{y_1^{q-1}-1,\dots,y_s^{q-1}-1\}$, then $G\neq 0$. Recall from above that under this procedure every monomial $\textbf{x}^{\textbf {a}}$ in $F$ yields the monomial $\textbf{y}^{\overline{Q\textbf{a}}}=y_1^{\overline{Q_1\textbf{a}}}\cdots y_s^{\overline{Q_s\textbf{a}}}$ whose total degree is $\deg(\textbf{y}^{\overline{Q_\textbf{a}}})=\overline{Q_1\textbf{a}}+\cdots+\overline{Q_s\textbf{a}}$. Thus, the total degree of $G$ is at most the mysterious number $d(\aa,Q)$ defined earlier. Therefore $G$ has at most $d(\alpha,Q)(q-1)^{s-1}$ roots in $(\K^*)^s$ by Lemma \ref{L:numberzero}.
	
	For any point $[P]=[\t^{\q_1}:\cdots:\t^{\q_r}]\in Y_Q$, we observe the following substantial property $$[P]\in V_X(F)\:\iff\:G(t_1\dots,t_s)=0,\:\forall\:(t_1\dots,t_s)\in \varphi_Q^{-1}([P]) $$ which implies that $|V_X(F)\cap Y_Q|=\frac{|V(G)\cap(\K^*)^s|}{|\textbf{H}|}$. Then, it follows immediately that $$|V_X(F)\cap Y_Q|\leq\frac{d(\alpha,Q)(q-1)^{s-1}}{|\textbf{H}|}.$$ Thus, the number $\mbox{max}\{|V_X(F)\cap Y_Q|\:|\:F\in S_\alpha\setminus I(Y_Q)\}$ being at most $\frac{d(\alpha,Q)(q-1)^{s-1}}{|\textbf{H}|}$, we get our lower bound on $\delta({\cC}_{\aa,Y_Q})$ as we claim.
\end{proof}

\subsection{Toric codes on Hirzebruch Surfaces} \label{SubSec:codesOnHirzebruch} In this section, we compute main parameters of the toric code $\mathcal{C}_{\aa,T_X}$ obtained from Hirzebruch surfaces, where $\alpha=(c,d)\in \N\beta$. Hansen computed these parameters for the case $c<q-1$ and $d=b$, where $b$ is to be defined below, see \cite{Ha1}.
 \begin{tm}\label{T:codesOnHirzebruch} Let $T_X$ be the torus of the Hirzebruch surface $ \mathcal{H_\ell}$ over $\K$ and $\alpha=(c,d)\in \N\beta$. Then the dimension of  toric code $\mathcal{C}_{\aa,T_X}$  is given by
 % If $c<q-1$, then the dimension of  toric code $\mathcal{C}_{\aa,T_X}$ is $\sum\limits_{a_4=0}^{b}(c+1-\ell a_4)=(a+1)(b+1)+\ell\frac{b(b+1)}{2}$ and its  minimum distance is given by $(q-1)(q-1-c)$  where $b$ be the greatest non-negative integer with the property that $\alpha=b(\ell,1)+(a,a')$ for some non-negative integers $a=c-b\ell\geq0$ and $a'=d-b\geq0$. Assume that $c\geq q-1 $ and $b\leq q-2$. The toric code ${\cC}_{\aa,Y_Q}$ has dimension equal to
%  $(q-1)(b'+1)+\sum\limits_{a_4=b'+1}^{b}(c-\ell a_4+1)$, minimum distance equal to  $(q-1)-b'$ where $b'$ be the greatest non-negative integer such that $c-b'\ell\geq q-2\geq0$ and $a'=d-b\geq0$.
 $$
 \mbox{dim}_{\K}{\cC}_{\aa,T_X} =
 \begin{cases}
(b+1)(c+1-\ell b/2), & \text{if }c<q-1 \\
 (q-1)(b'+1)+(b-b')(c+1-\ell(b+b'+1)/2, & \text{if }c\geq q-1\text{ and }b\leq q-2\\
   (q-1)(b'+1)+(q-2-b')(c+1-\ell(q-2+b'+1)/2, & \text{if }c\geq q-1,b > q-2\text{ and }b'<q-2\\
   (q-1)^2, & \text{if }c\geq q-1 \text{ and } b' \geq q-2\\
 \end{cases}
 $$ and  its  minimum distance equals
 $$
 \delta({\cC}_{\aa,T_X}) =
 \begin{cases}
 (q-1)(q-1-c), & \text{if }c<q-1 \\
 (q-1)-b', & \text{if }c\geq q-1\text{ and }b\leq q-2\\
 (q-1)-b', & \text{if }c\geq q-1,b > q-2\text{ and }b'<q-2\\
 1, & \text{if }c\geq q-1 \text{ and } b' \geq q-2\\
 \end{cases}
 $$ 
 where   $b$ (respectively $b'$) is the greatest non-negative integer with the property that $c-b\ell\geq 0\;\mbox{and}\;d-b\geq 0$ (respectively $c-b'\ell\geq q-2\;\mbox{and}\;d-b'\geq 0$). 
    \end{tm} 
   \begin{proof} %Consider  torus  $T_X$ of the Hirzebruch surface $ \mathcal{H_\ell}$ over $\K$ parameterized by identity matrix such that $\ell$ is positive integer. 
   	We first show that  $Q=\begin{bmatrix}~~0 & ~~0 & ~~1&~~0\\~~0 & ~~0 & ~~0 & ~1
	\end{bmatrix}$  parameterizes the torus, that is, $Y_Q=T_X.$ 
	$$\begin{bmatrix}h_1&h_2\end{bmatrix}\begin{bmatrix}~~0 & ~~0 & ~~1&~~0\\~~0 & ~~0 & ~~0 & ~1
	\end{bmatrix}\begin{bmatrix}~~1 & ~~0 & ~-1&~~0\\~~0 & ~~1 & ~~\ell & ~-1
	\end{bmatrix}^T=\begin{bmatrix}~~h_1\\~~h_1\ell-h_2
	\end{bmatrix}\equiv 0\;\mbox{mod}\;q-1,\:\mbox{for}\;0\leq h_1,h_2\leq q-2$$
	implies that $h_1=0=h_2$. So, $\textbf{H}=\{\hh \in \square_q \cap \Z^s \quad|\quad \hh Q\phi\equiv 0\;\mbox{mod}\;q-1\}=\{(0,0)\}$ and $|Y_Q|=(q-1)^2/|\textbf{H}|=(q-1)^2$. As $Y_Q\subset T_X$ and $|T_X|=(q-1)^2$, we have $Y_Q=T_X$, for $X=\mathcal{H_\ell}$.
	
	Let us find  a $\K-$basis for $S_\alpha$ for any $\alpha=(c,d)\in \N\beta$ where $\beta=\begin{bmatrix}~~1 & ~~0 & ~~1&~~\ell\\~~0 & ~~1 & ~~0 & ~1
	\end{bmatrix}$. Since $b$ is the greatest non-negative integer with the property that $\alpha=(c,d)=b(\ell,1)+(a,a')$ for some non-negative integers $a=c-b\ell\geq 0$ and $a'=d-b\geq 0$, the set $B_\alpha:=\{\textbf{x}^\textbf{a}\:|\:\mbox{deg}(\textbf{x}^\textbf{a})=\beta \textbf{a}=\alpha,\:0\leq a_4\leq b\}$ is a $\K-$basis for $S_\alpha$. For a fixed $a_4$, the power $a_2=d-a_4$ is fixed too and $a_1+a_3=c-\ell a_4$. So, $$B_\alpha=\{\textbf{x}^\textbf{a}\:|\:(a_1+a_3+\ell a_4,a_2+a_4)=\alpha,\:0\leq a_4\leq b,\:a_2=d-a_4,\:a_1+a_3=c-\ell a_4\}.$$ This means that for every choice of $a_4$ there are $c-\ell a_4+1$  possibilities for the tuple $(a_1,a_3)$, hence $$|B_\alpha|=\sum\limits_{a_4=0}^{b}(c+1-\ell a_4)=(c+1)(b+1)-\ell\frac{b(b+1)}{2}=(b+1)(c+1-\ell b/2).$$ 
	We know that columns of $\phi$ form a basis for $ L_\beta$ from Example \ref{ex:Hirzebruch}. Since $I(T_X)=I_{(q-1)L_\beta}$, columns of $\verb|ML|$ constitute a basis of $(q-1)L_\beta$, where
	$$\verb|ML|=(q-1)\phi=\begin{bmatrix}~q-1 & ~~~0 & -(q-1)&~~~0\\~~~0 & -(q-1) & -\ell(q-1) & ~(q-1)
	\end{bmatrix}^T.$$ 
Since $\verb|ML|$ is mixed dominating, it follows from Theorem \ref{t:mixeddominating} that $$I(T_X)=\la x_1^{q-1}-x_3^{q-1},x_4^{q-1}-x_2^{q-1}x_3^{\ell(q-1)}\ra.$$
	 Therefore $x_1^{q-1}= x_3^{q-1},x_4^{q-1}=x_2^{q-1}x_3^{\ell(q-1)}$ in the ring $S/I(T_X)$ and a basis for $S_{\aa}/I_{\aa}(T_X)$ is  $$\bar{B}_\alpha=\{\textbf{x}^\textbf{a}\:|\:\: a_1=c-a_3-\ell a_4, a_2=d-a_4, 0\leq a_3\leq \min \{c-\ell a_4,q-2\} \text{ \and } 0\leq a_4\leq \min \{b,q-2\}\}.$$
By the definition of $b'$, we have $\min \{c-\ell a_4,q-2\}=c-\ell a_4$ for $b'<a_4$ and $\min \{c-\ell a_4,q-2\}=q-2$ for $0\leq a_4 \leq b'$. The length of the code ${\cC}_{\aa,T_X}$ is $N=|T_X|=(q-1)^2.$ Next, we compute its dimension and minimum distance.
	
   \textbf{Case I:} Let $c=a+b\ell<q-1$. It is easy to see that $B_\alpha=\bar{B}_\alpha$ , so $\mbox{dim}_\K(\mathcal{C}_{\aa,T_X})=|B_\alpha|$.  Since $$d(\alpha,Q)=\mbox{max}\{\overline{Q_1\textbf{a}}+\overline{Q_2\textbf{a}}\:|\:\textbf{x}^\textbf{a}\in B_\alpha\}=\mbox{max}\{a_3+a_4\:|\:0\leq a_4\leq b,\:a_3+a_1=c-\ell a_4\}=c,$$  $\delta({\cC}_{\aa,Y_Q})\geq (q-1)^2-(q-1)c$ using Theorem \ref{t:mindislower}. On the other hand, for $\K^*=\la\eta\ra$, we have $$F=x_2^
    d\prod\limits_{i=1}^{c}(x_3-\eta^ix_1)\in S_\alpha$$ vanishing exactly at the $c(q-1)$ points $P_{i,j}=[1:1:\eta^i:\eta^j]\in T_X$, where $1\leq i \leq c$ and $1\leq j\leq q-1$. Thus,  there is a codeword $\mbox{ev}_{\alpha,T_X}(F)$ with weight $(q-1)^2-(q-1)c$. Hence,  $$\delta({\cC}_{\aa,T_X})=(q-1)^2-(q-1)c=(q-1)(q-1-c).$$ 
    
   \textbf{Case II:} Let $c\geq q-1 $ and $b\leq q-2$. Then $\min \{b,q-2\}=b$. So, if $0\leq a_4\leq b'$ then $0\leq a_3 \leq q-2$ but if $a_4 > b'$ then $0\leq a_3 \leq c-\ell a_4$, yielding the formula $$\mbox{dim}_{\K}{\cC}_{\aa,Y_Q}=|\bar{B}_\alpha|=(q-1)(b'+1)+\sum\limits_{a_4=b'+1}^{b}(c-\ell a_4+1) .$$ 
   
   Take  $F\in \bar{S}_\alpha$. Then we can write $$F(x_1,x_2,x_3,x_4)=\sum\limits_{a_4=0}^{b'}\left[\sum\limits_{a_3=0}^{q-2}k_{a_{3}a_4}x_3^{a_3}x_1^{c-\ell a_4-a_3}\right] x_4^{a_4}x_2^{d-a_4}+\sum\limits_{a_4=b'+1}^{b}\left[\sum\limits_{a_3=0}^{c-\ell a_4}k'_{a_{3}a_4}x_3^{a_3}x_1^{c-\ell a_4-a_3}\right]x_4^{a_4}x_2^{d-a_4}.$$
    For  any $G=G(y_3,y_4)=F(1,1,y_3,y_4),$ we set
       $$A=\{s_0\in\K^*\:|\:y_4-s_0\:\mbox{divides} \: G(y_3,y_4)\}$$
       and $V^*(G)=V(G)\cap(\K^*)^2.$
    The sets $V(G)\cap(\K^*\times A)$ and  $V(G)\cap(\K^*\times (\K^*\setminus A))$ form a partition of $V^*(G)$. Since $V(G)\cap(\K^*\times A)=a(q-1)$ and $V(G)\cap(\K^*\times (\K^*\setminus A))\leq d_3(q-1-a)$, we get
   \begin{equation}\label{e:min1}  |V^*(G)|\leq  |A|(q-1)+ d_3(q-1-|A|)     \end{equation}
   where $d_3=\mbox{deg}_{y_3}(G)$ and $a=|A|.$
   We claim that $|V^*(G)|\leq (q-1)(q-2)+b'$. $a$ is at most $b$, because $$\mbox{max}\{\mbox{deg}_{y_4}G\:|\:G(y_3,y_4)=F(1,1,y_3,y_4),F\in \bar{S}_\alpha\}=b.$$  Then there are three cases, depending upon the value of $a$:  $a\leq b',\:b'<a<b,\:a=b.$ \\
   
    \textbf{Case II.a:} We begin with the case $a\leq b'.$ This implies $d_3\leq q-2,$ because $d_3\leq c-b'\ell$ and $ c-b'\ell\geq q-2$. Then by (\ref{e:min1}), we conclude that
   $$|V^*(G)|\leq  a(q-1)+d_3(q-1-a)=a(q-1)+(q-2)(q-1-a)=(q-2)(q-1)+a\leq(q-2)(q-1)+b'.$$
   
   \textbf{Case II.b:} Suppose that $a=b'+k<b$ where $k\geq 1$ and $b'\neq b$ implies $d_3\leq c-(b'+k)\ell.$ From here,  the inequality (\ref{e:min1}) gives the following upper bound: 
   
  \begin{align}\nonumber
 | V^*(G)|& \leq  a(q-1)+(q-1-a)(c-\ell(b'+k))\nonumber\\
   &= a(q-1)+(q-1)(c-\ell(b'+k))-a(c-\ell(b'+k)\nonumber\\
   &= a(q-1)+(q-1)\left(c-\ell(b'+k)-(q-2)+(q-2)\right)-a(c-\ell(b'+k))\nonumber\\
   &=(q-1)(q-2)+(q-1)\left(c-\ell(b'+k)-(q-2)\right)+a\left(q-2-(c-\ell(b'+k))+1\right)\nonumber\\
 &=(q-1)(q-2)+\left(q-2-(c-\ell(b'+k))\right)\left(a-(q-1)\right)+a{} \label{e:min2}
   \end{align}
    On the other hand, we claim that $ c-(b'+1)\ell\leq q-3.$ To prove this, assume that $c-(b'+1)\ell\geq q-2$. Then $b'\neq b$ implies that $d-(b'+1)\geq 0$. So  this contradicts that $b'$ is the greatest non-negative integer with the property that $c-b'\ell\geq q-2\;\mbox{and}\;d-b'\geq 0$.\\
    It follows that $c-\ell(b'+k)=c-\ell(b'+1)-\ell(k-1)\leq q-3-\ell(k-1),$ which easily gives that $q-2-\left(c-\ell(b'+k)\right)\geq 1+ \ell(k-1)$. From $a<b\leq q-2$, we obtain that $a-(q-1)\leq-2$. If we combine the last two inequalities and (\ref{e:min2}), then we have
    \begin{equation}\label{e:min3}|V^*(G)|\leq(q-1)(q-2)-2\left(\ell(k-1)+1\right)+b'+k. \end{equation}
    Furthermore, $\ell\geq 2$ and $k-1\geq 0$, hence we get that $-2\left((k-1)\ell+1\right)\leq-4k-2$. Then (\ref{e:min3}) becomes
    $$|V^*(G)|\leq(q-1)(q-2)-3k-2+b'< (q-1)(q-2)+b',$$ 
    as required. 
   
    \textbf{Case II.c:} Consider the case $a=b\neq b'.$ Similar to (\ref{e:min2}), $d_3$ is bounded by $d_3\leq c-b\ell$ which together with  (\ref{e:min1}) gives 
       \begin{equation}\label{e:min4}| V^*(G)|\leq (q-1)(q-2)+\left(q-2-(c-\ell(b))\right)\left(b-(q-1)\right)+b.
      \end{equation} 
    Note that $c-\ell b=c-\ell\left(b'+1+b-(b'+1)\right)\leq q-3-\ell(b-b'-1)$, since $c-(b'+1)\ell\leq q-3$ proved in  Case II.b. Moreover, $b-(q-1)\leq-1$, thus we have 
     $$| V^*(G)|\leq (q-1)(q-2)+(\ell-1)(b'-b)+\ell -1+b'.$$
   From $b'<b$, we have $(\ell-1)(b'-b)\leq(\ell-1)(-1)=1-\ell$, therefore
   $$| V^*(G)|\leq (q-1)(q-2)+(\ell-1)(b'-b)+\ell -1+b'\leq(q-1)(q-2)+b'$$
   completing the proof of the claim.\\
    Thus, the minimum distance is at least  $(q-1)^2-(q-1)(q-2)-b'=(q-1)-b'$, since $$|V_\mathcal{H_\ell}(F)\cap T_X|=|V(F(1,1,y_3,y_4))\cap(\K^*)^4|=|V^*(G)|$$ for any $F\in \bar{S}_\alpha$. On the other hand,  since $|V^*(G_0)|=(q-1)(q-2)+b'$ for the polynomial $$G_0(y_3,y_4)=\prod\limits_{i=1}^{q-2}(y_3-\eta^i)\prod\limits_{j=1}^{b'}(y_4-\eta^j),$$ there is a codeword with  weight $(q-1)-b'$. This shows that
	$\delta({\cC}_{\aa,T_X})=(q-1)-b'$. Note that $G_0=F_0(1,1,y_3,y_4)$ for $$F_0(x_1,\dots,x_4)=x_1^{c-\ell b'-(q-2)}x_2^{d-b'}\prod\limits_{i=1}^{q-2}(x_3-\eta^ix_1)\prod\limits_{j=1}^{b'}(x_4-\eta^j x_1^{\ell}x_2).$$

	 \textbf{Case III:} Suppose $c\geq q-1 $, $b\geq q-2$  and $b'< q-2$. Since $0\leq a_4\leq \min \{b,q-2\}= q-2$, we get
	 
	 $$\mbox{dim}_{\K}{\cC}_{\aa,Y_Q}=|\bar{B}_\alpha|=(q-1)(b'+1)+\sum\limits_{a_4=b'+1}^{q-2}(c-\ell a_4+1) .$$

	  Pick $F\in \bar{S}_\alpha$. Then we can write $$F(x_1,x_2,x_3,x_4)=\sum\limits_{a_4=0}^{b'}\left[\sum\limits_{a_3=0}^{q-2}k_{a_{3}a_4}x_3^{a_3}x_1^{c-\ell a_4-a_3}\right] x_4^{a_4}x_2^{d-a_4}+\sum\limits_{a_4=b'+1}^{q-2}\left[\sum\limits_{a_3=0}^{c-\ell a_4}k'_{a_{3}a_4}x_3^{a_3}x_1^{c-\ell a_4-a_3}\right]x_4^{a_4}x_2^{d-a_4}.$$
	  The idea is the same as in Case II. To prove that  $\delta({\cC}_{\aa,T_X})=(q-1)-b'$, we show that the maximum value of $|V(G)|$ is $(q-1)(q-2)+b'$. In this case, $a\leq q-2$ and  we split the proof of the claim into three  cases:  $a\leq b',\:b'<a<q-2,\:a=q-2.$ The proof  is quite similar to that of the claim in Case II, so the proof is omitted here. 
	  
	  \textbf{Case IV:} Let $c\geq q-1 $ and $b'\geq q-2$. From $b'\leq b,$ $0\leq a_4\leq \min \{b,q-2\}= q-2$, We have
	 $$\bar{B}_\alpha=\{\textbf{x}^\textbf{a}\:|\:\: a_1=c-a_3-\ell a_4, a_2=d-a_4, 0\leq a_3\leq q-2 \text{ \and } 0\leq a_4\leq q-2\}$$
	 giving  $\mbox{dim}_{\K}{\cC}_{\aa,Y_Q}=|\bar{B}_\alpha|=(q-1)^2 .$ The code ${\cC}_{\aa,Y_Q}$ is trivial, that is, $\delta({\cC}_{\aa,T_X})=1.$
	\end{proof}

\begin{rema} As the referee pointed out, it is noteworthy that the same polynomials are used independently to give some codewords having the minimum weight in Theorem \ref{T:codesOnHirzebruch} and in \cite[Proposition 4.2.4]{Nardi}. 
\end{rema}

\begin{ex} \label{T:ParcodesOnHirzebruch} 
Here, we give another family of codes whose minimum distance attains our bound. These are actually Reed-Solomon codes obtained from cyclic subgroups of the torus $T_X$, for the Hirzebruch surface $X=\cl H_{\ell}$ over $\K=\F_{q}$,  where $q$ is an odd prime power. 

Take $Q=[q_1\quad  q_2\quad  q_1+2\quad  q_1\ell+q_2]$ with $q_1,q_2\in\Z$ and $\alpha=(c,d)\in \N\beta$. Then, we show below that the parameterized code ${\cC}_{\aa,Y_Q}$ is a non trivial MDS code with parameters $[\frac{q-1}{2},c+1,\frac{q-1}{2}-c]$  if $ c<\frac{q-1}{2}$ and is a trivial code otherwise.

   	First recall from Example \ref{Ex:CI} that $Y_Q=Y_Q'$ where $Q'=[0\;  0\;  2\;  0]$. It follows that  $Y_Q$ is generated by the point $[1:1:\eta^2:1]$ with $\eta$ being a generator for $\K^*$. Hence the order of $Y_Q$ equals $|\eta^2|=\frac{q-1}{(q-1,2)}=\frac{q-1}{2}$ proving that $|\textbf{H}|=2$ and that the length of ${\cC}_{\aa,Y_Q}$ is $N=\frac{q-1}{2}$. 
   	
   	Example  \ref{Ex:CI} gives also that $I(Y_Q)=I_L$ is  generated by $\ x_1^2x_2-x_4$ and  $x_1^{(q-1)/2}-x_3^{(q-1)/2} $. Thus, $x_4\equiv x_1^{\ell}x_2$ and $ x_3^{\frac{q-1}{2}}\equiv x_1^{\frac{q-1}{2}} $ in $ S/\mathrm{I}_{L} $. Hence, a basis for the vector space $S_{\aa}/I_{\aa}(Y_Q)$ is given by  
   	$$ \bar{B}_\alpha=\begin{cases} \{x_1^{c-a_3}x_2^{d} x_3^{a_3} \:|\:    0\leq a_3 \leq c\} & \mbox{if } c<\frac{q-1}{2} \\ 
       \{x_1^{c-a_3}x_2^{d} x_3^{a_3} \:|\:  0\leq a_3 < (q-1)/2\} & \mbox{if }  c \geq\frac{q-1}{2}
      \end{cases}
$$ leading to
   $$\mbox{dim}_{\K}{\cC}_{\aa,Y_Q}=|\bar{B}_\alpha|= \begin{cases} c+1 & \mbox{if } c<\frac{q-1}{2} \\ 
       \frac{q-1}{2} & \mbox{if }  c \geq\frac{q-1}{2}.
      \end{cases} $$

%    Let $\aa\in \N\beta$. Now compute the dimension of ${\cC}_{\aa,Y_Q}$. As noted in the proof of Theorem \ref{T:codesOnHirzebruch}, 
%   $$B_\alpha=\{\textbf{x}^\textbf{a}\:|\:(a_1+a_3+2 a_4,a_2+a_4)=\alpha,\:0\leq a_4\leq b,\:a_2=d-a_4,\:a_3+a_1=c-2 a_4\}.$$

 To compute the minimum distance of ${\cC}_{\aa,Y_Q}$, assume that  $c<\frac{q-1}{2}$ as otherwise the dimension reaches its maximum value revealing that the code is trivial. Theorem \ref{t:mindislower}  implies that  the minimum distance of  ${\cC}_{\aa,Y_Q}$ is at least $ N-c $, because $|\textbf{H}|=2$ and $d(\alpha,Q')=\mbox{max}\{2a_3\:|\:0\leq a_3 \leq c\}=2c.$
 
On the other hand, the polynomial $F=x_2^d \prod\limits_{j=1}^{c}(x_3-\eta^{2j} x_1)$ has exactly $c$ roots, verifying that $\delta({\cC}_{\aa,Y_Q})$ is at most $N-c$. Thus, the minimum distance reaches the Singleton bound $ \delta({\cC}_{\aa,Y_Q})=N-c= N+1-K$ and hence the code is MDS.
\end{ex}

\begin{ex}\label{T:parametersweightedtorus} We give yet another instance where our bound on minimum distance is attained by codes obtained from cyclic subgroups of the torus of a weighted projective space introduced in Example \ref{ex:weighted}. Recall that the homogeneous coordinate ring of the weighted projective space $X=\Pp(1,w_1,\dots, w_n)$ over   $\K=\F_q$ is $\K[x_0,x_1,\dots,x_n]$ which is $\Z$-graded where $\deg_{\cA} (x_0)=1$ and $\deg_{\cA}(x_i)=w_i>0$ for $i=1,\dots,n$. 

We also recall that a $\Z$-basis for the key lattice $L_\beta$ for the row matrix $\beta=[1\; w_1\; \cdots\; w_n],$ is given by $\{\textbf{u}_1, \dots, \uu_n \}\subset \Z^{n+1}$ where $\textbf{u}_i=(-w_i,\textbf{e}_i)$ and $\textbf{e}_i$ is the standard basis vector of $\Z^n$, for each $i=1,\dots,n$. Let $Q=[0\;|\;a\textbf{e}_i]$ be the row matrix with a unique nonzero positive integer $a$ at the $i$-th column for $i\in \{1,\dots n\}$ together with $n$ zero columns elsewhere. Assume that $a$ divides $q-1$.

    At first glance it is not clear that the code ${\cC}_{\aa,Y_Q}$ is a Reed-Solomon code. This will be clear from the set $\bar{B}_\alpha$ obtained below using our results.
    
    Let $\alpha(i)$ be the greatest  non-negative integer to satisfy  $\alpha=\alpha(i)w_i+\alpha'(i)$ for some $0\leq\alpha'(i)<w_i$.

    It is easy to see that the point $[1:\cdots:\eta^a:\cdots:1]$ with $\eta^a$ at the $i$-th component generates $Y_Q$, and that $|Y_Q|=|\eta^a|=\frac{q-1}{a}$ and hence $|\textbf{H}|=a$.
    
    Let us find generators for the vanishing ideal of $Y_Q\subset \K[x_0,x_1,\dots,x_n]$. By Lemma \ref{lm:equality}, $I(Y_Q)=I_L$ for the lattice $L=\{\m\in L_\beta: Q\m\equiv 0\;\mbox{mod}\;(q-1)\}$. Using Example \ref{ex:Hirzebruch},we find a basis for $L$. Take $\m\in L_\beta$, then $\m=\phi \textbf{c} $ for some $\textbf{c}\in \Z^n$. $\m\in L$ if and only if $Q\phi\textbf{c}=a\textbf{e}_i\textbf{c}=ac_i\equiv 0\;\mbox{mod}\;q-1$ which is equivalent to $c_i=\frac{q-1}{a}k$ for some $k\in \Z$. Thus, $\m\in L$ if and only if $$\m=c_1 \uu_1+\cdots +c_i \uu_i+\cdots+c_n \uu_n=c_1 \uu_1+\cdots +k(\frac{q-1}{a} \uu_i)+\cdots+c_n \uu_n.$$
    Hence, a $\Z$-basis for $L$ is given by the set $\{\textbf{u}_1, \dots,\uu_{i-1},\frac{q-1}{a} \uu_i,\uu_{i+1},\dots, \uu_n \}$. Since the matrix $\verb|ML|=[\textbf{u}_1 \cdots \uu_{i-1} \quad \frac{q-1}{a} \uu_i \quad \uu_{i+1} \cdots \uu_n]$ is mixed dominating, $I(Y_Q)$ is a complete intersection generated by $$\{ F_1,\dots,F_i,\dots,F_n\} \text{ where } F_i=x_0^{(q-1)w_i/a}-x_i^{(q-1)/a} \text{ and } F_j=x_0^{w_j}-x_j \text{ for } j\in\{1,\dots,n\}\setminus\{i\}.$$ Therefore, $x_0^{(q-1)w_i/a}=x_i^{(q-1)/a}$ and $x_0^{w_j}=x_j$ for $j\in\{1,\dots,n\}\setminus\{i\}$ in the quotient ring $S/I(Y_Q)$. For a positive integer  $\alpha\in\N\beta=\N$, a basis $\bar{B}_\alpha$ for the vector space $S_{\aa}/I_{\aa}(Y_Q)$ is given by 
    $$ \bar{B}_\alpha= \begin{cases}  \{x_0^{\aa-a_iw_i}x_i^{a_i}\:|\:0\leq a_i\leq\alpha(i)\} & \mbox{if } \alpha(i)<\frac{q-1}{a} \\ 
	\{x_0^{\aa-a_iw_i}x_i^{a_i}\:|\:0\leq a_i < \frac{q-1}{a}\} & \mbox{if }  \alpha(i) \geq\frac{q-1}{a}.
	\end{cases}$$
So, we get 
	$$  \mbox{dim}_{\K}{\cC}_{\aa,Y_Q}=H_{Y_Q}(\alpha)=|\bar{B}_\alpha|= \begin{cases} \alpha(i)+1 & \mbox{if } \alpha(i)<\frac{q-1}{a} \\ 
	\frac{q-1}{a} & \mbox{if }  \alpha(i) \geq\frac{q-1}{a}.
	\end{cases} $$\\
	It is now time to show that ${\cC}_{\aa,Y_Q}$ is an MDS code of length $N=\frac{q-1}{a}$. If  $ \alpha(i) \geq\frac{q-1}{a}$ then the dimension reaches its upper bound, i.e.  $\mbox{dim}_{\K}{\cC}_{\aa,Y_Q}=\frac{q-1}{a}=N$, so that the code ${\cC}_{\aa,Y_Q}$ is a trivial MDS code. 
	
	Assume now that $ \alpha(i) < \frac{q-1}{a}$. Then   by Theorem \ref{t:mindislower}, we have  $\delta({\cC}_{\aa,Y_Q})\geq \frac{q-1}{a}-\alpha(i)$, because $$d(\alpha,Q)=\mbox{max}\{a_ia\:|\:0\leq a_i\leq \alpha(i)\}=a\alpha(i).$$ 
	Consider the polynomials $F\in S_\alpha$ and $G(y)=F(1,\dots,y^a,\dots,1)\in \K[y]$, where $$F=x_0^{\aa-a_iw_i}x_i^{a_i-\aa(i)}\prod\limits_{j=1}^{\alpha(i)}(x_i-(\eta^a)^jx_0^{w_i}) \text{ and } G(y)=y^{a(a_i-\alpha(i))}\prod\limits_{j=1}^{\alpha(i)}(y^a-(\eta^a)^j).$$ The polynomial $G$ has exactly $a\alpha(i)$ roots, as $y^a=(\eta^a)^j$ has exactly $a$ solutions $(\eta^{\frac{q-1}{a}})^k\eta^j$ corresponding to integers $k\in[0, a-1]$ for every $j\in\{1,\dots,\alpha(i)\}$. Thus, $$|V_X(F)\cap Y_Q|=\frac{|V(G)\cap \K^*|}{|\textbf{H}|}=\alpha(i)$$ and $F$ produces the codeword with weight $\frac{q-1}{a}-\alpha(i)$, that is,  $\delta({\cC}_{\aa,Y_Q})= \frac{q-1}{a}-\alpha(i)=N-K+1$.
	
	Finally, we prove that  for any $\alpha_1,\alpha_2\in \N\beta $, the codes ${\cC}_{\alpha_1,Y_Q}$ and ${\cC}_{\alpha_2,Y_Q}$ are equivalent as soon as $\alpha_1(i)=\alpha_2(i)$. This will follow from \cite[Proposition 4.3]{sasop} since in this case both codes have the same dimensions: $H_{Y_Q}(\alpha_1)=H_{Y_Q}(\alpha_2)$ and that we have either $\alpha_1-\alpha_2 \in \N\beta=\N$ or $\alpha_2-\alpha_1\in\N\beta=\N$.
\end{ex}

\section{Examples} \label{S:Ex}

In this section, we give examples to reveal that some toric varieties other than $\Pp^n$ can have more and better codes, and to demonstrate that certain subgroups $Y_Q$ of $T_X$ can produce better codes than $T_X$ produces. 

Let us start by explaining what we mean from "better" in this context. A classical approach to compare two codes having the same length and dimension is to compare the remaining parameter: the minimum distance. The code with a bigger minimum distance is regarded better as it will have a bigger error-correction capacity. A code is called BP (best possible) if its minimum distance attains the maximum possible value among all codes with the same length and dimension, which can be checked online using the database \cite{grassl} recording lower and upper bounds for the minimum distance. For a given length and dimension, this database lists a BK (best known) code over a finite field with at most $9$ elements, whose minimum distance determines the lower bound. Although the upper bound is theoretical, it does not come from the same source for all the codes. A unique but mostly weaker bound also known as the Singleton's bound is given by $\delta\leq N+1-K$ for a given code with parameters $[N,K,\delta]$. A code is called MDS (maximum distance separable), if $\delta$ attains its maximum value, i.e. $\delta=N+1-K$. A primary goal of the coding theory is to improve the lower bound by exhibiting new codes with a higher minimum distance beating the BK code as well as to demonstrate the existence codes whose minimum distance reaches the upper bound in \cite{grassl}. Toric codes have been used to produce such champion codes. The techniques of this paper can be used for a systematic search for obtaining new champion codes. 

As we evaluate homogeneous polynomials of degree $\aa$ on a subgroup $Y_Q$ of the torus $T_X$, the code $\mathcal{C}_{\aa,Y_Q}$ is a \textit{puncturing} of the toric code $\mathcal{C}_{\aa,T_X}$. In order to compare two such codes whose lengths or dimensions are different, we use the following approach. The Singleton's bound implies the inequality $N+1-\delta-K\geq 0$ for a code $\mathcal{C}$ with parameters $[N,K,\delta]$. This inequality is clearly equivalent to $S(\mathcal{C}) \geq 0$, where $$S(\mathcal{C})=(1-K/N)-(\delta-1)/N.$$
Thus, a code will be regarded \textit{better} if it has a smaller $S(\mathcal{C})$ value. Notice that the code is MDS if and only if $S(\mathcal{C})=0$. When two codes have the same $S(\mathcal{C})$ value, there is another invariant $E(\mathcal{C})$, which measures the extend to which the code approaches its error correction capacity, where $$E(\mathcal{C})=(1-K/N)-\lfloor (\delta-1)/2 \rfloor /N  .$$ Then, a \textit{better} code will have a smaller $E(\mathcal{C})$.

Using the techniques  developed in the previous sections we first identify the finite list of non-trivial and non-equivalent codes in the following $2$ examples. Then, we compute parameters of these codes and make Table \ref{codes} and Table \ref{codes2} for comparing them. There are only $2$ codes from Example \ref{Ex:P3}, where $Y_Q \subset \Pp^3$. We use the same $Q$ in Example \ref{Ex:YQH2}, where $Y_Q \subset \cl H_2$ and get $6$ codes. There are three BP and one MDS codes among them. This indicates that \textit{considering different toric varieties $X$ as ambient spaces is a good alternative for the projective space} $\Pp^n$, see Table \ref{codes}. 
\begin{longtable}[h!]{|c|c|c|c|c|}
	\caption{Code Comparison.\label{codes}}\\
	\hline
	{$\mathbf{\alpha}$ \cellcolor{gray!20}}& {$[N,K,\delta]$ \cellcolor{gray!20}}& {$S(\mathcal{C}_{\aa,Y_Q})$ \cellcolor{gray!20}}& {$E(\mathcal{C}_{\aa,Y_Q})$ \cellcolor{gray!20}}&{status \cellcolor{gray!20}} \\
	\hline
	\multicolumn{5}{| c |}{Codes on $Y_Q \subseteq \Pp^3$ \label{tab:1}}\\
	
	\hline
	1  &[16, 4, 9]& 1/4 & 1/2 & \\ 
	\hline 
	2 &[16, 9, 4] & 1/4& 3/8& \\ 
	\hline
	\multicolumn{5}{|c|}{Codes on $Y_Q \subseteq \cl H_2$}\\
	\hline
	(1, 0) &[8, 2, 6] & 1/8&1/2 & BP  \\ 
	\hline
	(2, 0) &[8, 3, 4] & 1/4&1/2 &\\ 
	\hline
	(3, 0) &[8, 4, 2] & 3/8&1/2 &   \\ 
	\hline
	(2, 1) &[8, 4, 4] & 1/8&3/8 & BP\\ 
	\hline
	(3, 1) &[8, 6, 2] & 1/8&1/4 & BP    \\ 
	\hline
	(4, 1) &[8, 7, 2] & 0&1/8 &MDS\\ 
	\hline
	
\end{longtable}	

 The set $Y_Q$ of Example \ref{Ex:YQH2} is a proper subgroup of $T_X$ and so $\mathcal{C}_{\aa,Y_Q}$ is a puncturing of the code $\mathcal{C}_{\aa,T_X}$, which is obtained by deleting the $8$ components of a code word in $\mathcal{C}_{\aa,T_X}$ corresponding to the points in $T_X\setminus Y_Q$. By looking at the Table \ref{codes2}, we see that $S(\mathcal{C}_{\aa,Y_Q}) < S(\mathcal{C}_{\aa,T_X})$ and that $E(\mathcal{C}_{\aa,Y_Q}) < E(\mathcal{C}_{\aa,T_X})$ for all degrees $\aa\in \{(1,0),(2,0),(3,0),(2,1),(3,1),(4,1)\}$. We know from Example \ref{Ex:YQH2} that the codes $\mathcal{C}_{\aa,Y_Q}$ are trivial MDS codes with parameters $[8,8,1]$ for $\aa\in \{(5,1),(4,2),(5,2),(6,2),(7,2),(6,3),(7,3),(8,3)\}$. So, $S(\mathcal{C}_{\aa,Y_Q})=E(\mathcal{C}_{\aa,Y_Q})=0$ and thus the code $\mathcal{C}_{\aa,Y_Q}$ is always better than the code $\mathcal{C}_{\aa,T_X}$. This reveals that \textit{puncturing the code} $\mathcal{C}_{\aa,T_X}$ \textit{by considering the proper subgroup $Y_Q$ of $T_X$ produces better codes}.

%\begin{table}[h!]--Tabloyu metinde nereye koyarsan orada görünür
%  \hline
% \multicolumn{5}{|c|c|c|}{Codes on $Y_Q \subseteq \cl H_2$}& \multicolumn{5}{| c |}{Codes on $T_X \subseteq X=\cl H_2$}\\
\begin{longtable}[h!]{|c|c|c|c|c||c|c|c|c|c|}
	\caption{Code Comparison on $X=\cl H_2$.\label{codes2}}\\
	\hline
	{$\mathbf{\alpha}$ \cellcolor{gray!20}}& {$[N,K,\delta]$ \cellcolor{gray!20}}& {$S(\mathcal{C}_{\aa,T_X})$ \cellcolor{gray!20}}& {$E(\mathcal{C}_{\aa,T_X})$ \cellcolor{gray!20}}&{status \cellcolor{gray!20}} & {$[N,K,\delta]$ \cellcolor{gray!20}}& {$S(\mathcal{C}_{\aa,Y_Q})$ \cellcolor{gray!20}}& {$E(\mathcal{C}_{\aa,Y_Q})$ \cellcolor{gray!20}}&{status \cellcolor{gray!20}}\\
	
	\hline
	(1, 0) &[16, 2, 12] & 3/16 &9/16 & & [8, 2, 6] & 1/8&1/2 & BP \\ 
	\hline
	(2, 0) &[16, 3, 8] & 3/8&5/8 & &[8, 3, 4] & 1/4&1/2 &\\ 
	\hline
	(3, 0) &[16, 4, 4]  & 9/16&11/16& & [8, 4, 2] & 3/8&1/2 &  \\ 
	\hline
	(2, 1) &[16, 4, 8] & 5/16&9/16 & &[8, 4, 4] & 1/8&3/8 & BP\\ 
	\hline
	(3, 1) &[16, 6, 4]  & 7/16 &9/16 & &[8, 6, 2] & 1/8&1/4 & BP   \\ 
	\hline
	(4, 1) &[16, 7, 4] & 3/8& 1/2 & &[8, 7, 2] & 0&1/8 &MDS\\
	\hline
	(5, 1) &[16, 8, 3]  & 3/8&7/16 & &[8, 8, 1] & 0&0 &MDS  \\ 
	\hline
	(4, 2) &[16, 8, 4] & 5/16& 7/16 & &[8, 8, 1] & 0&0 &MDS\\
	\hline
	(5, 2) &[16, 10, 3] & 1/4&5/16 & &[8, 8, 1] & 0&0 &MDS  \\ 
	\hline
	(6, 2) &[16, 11, 3] & 3/16&1/4 & &[8, 8, 1] & 0&0 &MDS\\
	\hline
	(7, 2) &[16, 12, 2] & 3/16&1/4 &  &[8, 8, 1] & 0&0 &MDS  \\ 
	\hline
	(6, 3) &[16, 12, 3] & 1/8&3/16 & &[8, 8, 1] & 0&0 &MDS\\
	\hline
	(7, 3) &[16, 14, 2] & 1/16& 1/8 & BP &[8, 8, 1] & 0&0 &MDS  \\ 
	\hline
	(8, 3) &[16, 15, 2] & 0&1/16 &MDS &[8, 8, 1] & 0&0 &MDS\\ 
	\hline
	
\end{longtable}	

We conclude the section by explaining how to list the codes filling in the Table \ref{codes} and Table \ref{codes2}. The values listed in the first part of Table \ref{codes2} were calculated by Theorem \ref{T:codesOnHirzebruch}. We start with the case of the projective space.

\begin{ex}[Codes on $Y_Q \subseteq \Pp^3$]\label{Ex:P3} Fix $q=5$ and consider the incidence matrix $Q$ of the square shaped graph with vertices $V=\{1,2,3,4\}$ and edges $E=\{\{1,2\},\{2,3\},\{3,4\},\{1,4\}\}$. Then, the parameterized subgroup of the torus in the projective space $\Pp^3$ is $$Y_Q=\{[t_1t_2:t_2t_3:t_3t_4:t_1t_4]~:~t_1,t_2,t_3,t_4\in \F_5\}\subseteq \Pp^3.$$
	Using the Procedure \ref{pdr:length}, we find that $N=|Y_Q|=16$ first. In the second step, using either of the algorithm \ref{a:lattice1}, we find the following $10$ minimal generators of $I(Y_Q)$:
	$$x_1x_3-x_2x_4, ~~x_1^4-x_4^4, ~~x_1^3x_2-x_3x_4^3, ~~x_1^2x_2^2-x_3^2x_4^2, ~~x_1x_2^3-x_3^3x_4,$$
	$$ 
	x_2^4-x_4^4, ~~~x_2^3x_3-x_1^3x_4, ~~~x_2^2x_3^2-x_1^2x_4^2, ~~~x_2x_3^3-x_1x_4^3, ~~~x_3^4-x_4^4.$$
	As in Example \ref{P(2,2,3,5)}, we see that the first $3$ values of the Hilbert function are ${1, 4, 9}$ and that the rest is $H_{Y_Q}(\aa)=16$, for $\aa>2$. This means that $Q$ produces only $2$ non-trivial codes forming the first part of Table \ref{codes}. 
\end{ex}
\begin{ex}[Codes on $Y_Q \subseteq \cl H_2$] \label{Ex:YQH2} We continue to use the same $q$ and $Q$ as in Example \ref{Ex:P3} and get more and better codes on the Hirzebruch surface $\cl H_2$. In this case, the size becomes $N=|Y_Q|=8$. By using the algorithms we develop in previous sections, we compute the following minimal generating set for $I(Y_Q)$:
	$$I(Y_Q)=\langle x_1^4-x_3^4, \quad x_1^2x_2^2x_3^2-x_4^2\rangle.$$ So, $Y_Q$ becomes a complete intersection on $\cl H_2$. Since $x_1^4-x_3^4, x_1^2x_2^2x_3^2-x_4^2\in I(Y_Q)$, it follows that $x_3^{4}+I(Y_Q)=x_1^{4}+I(Y_Q)$ and $x_4^{2}+I(Y_Q)=x_1^{2}x_2^{2}x_3^2+I(Y_Q)$ in the quotient ring $S/I(Y_Q)$. So, we have the following bases $\bar{B}_{\aa}$ for the vector space $S_{\aa}/I(Y_Q)_{\aa}$. 
	$$\bar{B}_{(1,0)}=\{x_1,x_3\};\quad \bar{B}_{(2,0)}=\{x_1^2,x_1x_3,x_3^2\};\quad \bar{B}_{(c,0)}=\{x_1^c,x_1^{c-1}x_3,x_1^{c-2}x_3^2,x_1^{c-3}x_3^3\}, \quad \mbox{for} ~~c>2,$$
	$$ \bar{B}_{(0,d)}=\{x_2^d\}\quad \mbox{and}  \quad \bar{B}_{(1,d)}=\{x_1x_2^d,x_3x_2^d\}, ~~\mbox{for} ~~d\in \N; \quad \bar{B}_{(2,d)}=\{x_1^2x_2^d,x_1x_3x_2^d,x_3^2x_2^d,x_2^{d-1}x_4\}~~\mbox{for} ~~d>0,$$
	$$\bar{B}_{(3,d)}=\{x_1^3x_2^d,x_1^2x_3x_2^d,x_1x_3^2x_2^d,x_3^3x_2^d,x_1x_2^{d-1}x_4,x_3x_2^{d-1}x_4\} ~~\mbox{for} ~~d>0,$$
	$$\bar{B}_{(4,d)}=\{x_1^4x_2^d,x_1^3x_3x_2^d,x_1^2x_3^2x_2^d,x_1x_3^3x_2^d,x_1^2x_2^{d-1}x_4,x_1x_3x_2^{d-1}x_4,x_3^2x_2^{d-1}x_4\} ~~\mbox{for} ~~d>0,$$
	$$\bar{B}_{(5,d)}=\{x_1^5x_2^d,x_1^4x_3x_2^d,x_1^3x_3^2x_2^d,x_1^2x_3^3x_2^d,x_1^3x_2^{d-1}x_4,x_1^2x_3x_2^{d-1}x_4,x_1x_3^2x_2^{d-1}x_4,x_3^3x_2^{d-1}x_4\} ~~\mbox{for} ~~b>0.$$
	Thus, the values of $H_{Y_Q}(c,0)$ starting from $c=0$ are $1, 2, 3, 4, 4, 4, 4, \dots$, and the values of $H_{Y_Q}(c,1)$ are $1, 2, 4, 6, 7, 8$ for $c=0,1,2,3,4,5$. By \cite[Corollary 3.18]{sasop}, if $\aa'-\aa\in \N\bb$ then $H_{Y_Q}(\aa)\leq H_{Y_Q}(\aa')$. Thus, we have $8=H_{Y_Q}(5,1)\leq H_{Y_Q}(a,1)\leq 8$, for all $c>5$, as $(c-5,0)\in \N\bb$. Similarly, we have $H_Y(c,d)=8$ for all $c>5$ and $d>0$, as $(c-5,d-1)\in \N\bb$. Hence, the values of $H_{Y_Q}(c,1)$ starting from $ac=0$ are $1, 2, 4, 6, 7, 8, \dots$ and this sequence of $H_{Y_Q}(c,d)$ is always the same, for any $d>1$ as $(0,d-1)\in \N\bb$. By \cite[Proposition 4.3]{sasop}, if $\aa'-\aa\in \N\bb$ and $H_{Y_Q}(\aa)=H_{Y_Q}(\aa')$ then the codes $\mathcal{C}_{\aa,Y_Q}$ and $\mathcal{C}_{\aa',Y_Q}$ are equivalent, i.e. have the same parameters. Hence, the only non-equivalent and non-trivial codes are the generalized toric codes $\cC_{\aa,Y_Q}$ for the degrees $\aa\in \{(1,0),(2,0),(3,0),(2,1),(3,1),(4,1)\}$. Notice that although $H_{Y_Q}(3,0)= H_{Y_Q}(2,1)=4$, the corresponding codes are not equivalent, which is not surprising as $\pm[(3,0)-(2,1)]\notin \N\bb$. For the parameters of the corresponding codes, see the second part of Table \ref{codes}.
\end{ex}

\section*{Acknowledgements} We thank O\u{g}uz Yayla for his valuable helps on Section \ref{S:Ex}. We also thank an anonymous referee for helpful comments and suggestions improving the presentation of the paper.

\bibliographystyle{amsplain}

\end{document}